\documentclass[11pt, reqno]{amsart}
\usepackage{amssymb,amsthm,amsmath}
\usepackage[numbers,sort&compress]{natbib}
\usepackage{amssymb,amsmath}
\usepackage{amsfonts}
\usepackage{mathrsfs}
\usepackage{latexsym}
\usepackage{amssymb}
\usepackage{amsthm}
\usepackage{indentfirst}
\let\oldsection\section
\renewcommand\section{\setcounter{equation}{0}\oldsection}

\newtheorem{theorem}{Theorem}[section]
\newtheorem{lemma}{Lemma}[section]
\newtheorem{proposition}{Proposition}[section]
\newtheorem{definition}{Definition}[section]

\newtheorem{remark}{Remark}[section]

\date{December 20, 2013}

\begin{document}

\title[Strong Solutions to the 3D Primitive Equations]{Local and Global Well-posedness of Strong Solutions to the 3D Primitive Equations With Vertical Eddy Diffusivity}

\author{Chongsheng~Cao}
\address[Chongsheng~Cao]{Department of Mathematics, Florida International University, University Park, Miami, FL 33199, USA}
\email{caoc@fiu.edu}

\author{Jinkai~Li}
\address[Jinkai~Li]{Department of Computer Science and Applied Mathematics, Weizmann Institute of Science, Rehovot 76100, Israel}
\email{jklimath@gmail.com}

\author{Edriss~S.~Titi}
\address[Edriss~S.~Titi]{Department
of Computer Science and Applied Mathematics, Weizmann Institute of Science,
Rehovot 76100, Israel. Also Department of Mathematics and Department of Mechanical and Aerospace Engineering, University of California, Irvine, California 92697-3875, USA}
\email{etiti@math.uci.edu and edriss.titi@weizmann.ac.il}

\begin{abstract}
In this paper, we consider the initial-boundary value problem of the viscous 3D primitive equations for oceanic and atmospheric dynamics with only vertical diffusion in the temperature equation. Local and global well-posedness of strong solutions are established for this system with $H^2$ initial data.
\end{abstract}

\maketitle



\allowdisplaybreaks
\section{Introduction}\label{sec1}

The primitive equations are derived from the full incompressible Navier-Stokes
equations using the Boussinesq and hydrostatic approximations. They are the fundamental models for weather prediction, see, e.g., Lewandowski \cite{LEWAN}, Pedlovsky \cite{PED}, and Washington and Parkinson \cite{WP}. In the context of the oceans and the atmosphere dynamics the horizontal scales are much larger than the vertical one. By taking advantage of this, the scale analysis (see, e.g., Pedlovsky \cite{PED} and Vallis \cite{VALLIS}) leads to the hydrostatic approximation, see also Az\'erad and Guill\'en \cite{AZERADGUILLEN} and Lions, Temam and Wang \cite{LTW92B} for the rigourous mathematical justification.

In this paper, we consider the primitive equations with only vertical diffusion. The primitive equations are given by the following system (see, e.g., \cite{LTW92A,LTW92B,TZ04,PTZ09,MAJDA})
\begin{eqnarray}
&\partial_tv+(v\cdot\nabla_H)v+w\partial_zv+\nabla_Hp+L_1v+f_0k\times v=0,\label{1.1}\\
&\partial_zp+T=0,\label{1.2}\\
&\nabla_H\cdot v+\partial_zw=0,\label{1.3}\\
&\partial_tT+(v\cdot\nabla_H)T+w\partial_zT+L_2T=Q,\label{1.4}
\end{eqnarray}
where the horizontal velocity $v=(v^1,v^2)$, the vertical velocity $w$, the temperature $T$ and the pressure $p$ are the unknowns, $f_0$ is the Coriolis parameter, and $Q$ is a given heat source. Here, for simplicity, we assume that the heat source $Q$ is identically
zero; however, the results obtained in this paper hold true for the nonzero but appropriately
regular $Q$. The operators $L_1$ and $L_2$ in (\ref{1.1}) and (\ref{1.4}) are the viscosity and the heat vertical diffusion
operators, respectively, given by
$$
L_1=-\frac{1}{R_1}\Delta_H-\frac{1}{R_2}\partial_z^2,\qquad L_2=-\frac{1}{R_3}\partial_z^2
$$
with positive constants $R_1, R_2$ and $R_3$, where $R_1, R_2$ represent the horizontal and vertical dimensionless Reynolds numbers, respectively, while $R_3$ is the vertical dimensionless
eddy heat diffusivity turbulence mixing coefficient (see \cite{GARGETT,GARRETT} for example). In this paper, we use the notations $\nabla_H=(\partial_x,\partial_y)$ and $\Delta_H=\partial_x^2+\partial_y^2$ to stand for the horizontal gradient and Laplacian, respectively.

The mathematical studies of primitive equations were initialed by Lions, Temam and Wang \cite{LTW92A,LTW92B,LTW95} in 1990s, where the global existence of weak solutions were obtained.
Weak solutions in 2D turn out to be unique, see Bresch, Guill\'en-Gonz\'alez, Masmoudi and Rodr\'iguez-Bellido \cite{BGMR03}; however, the uniqueness of weak solutions in the three-dimensional case is still unclear. Concerning the strong solutions for the 2D case, the local existence result was established by Guill\'en-Gonz\'alez, Masmoudi and Rodr\'iguez-Bellido \cite{GMR01}, while the global existence for 2D case was achieved  by Bresch, Kazhikhov and Lemoine in \cite{BKL04} and Temam and Ziane in \cite{TZ04}. The global existence of strong solutions for 3D case was established by Cao and Titi \cite{CAOTITI2}. In \cite{CAOTITI2}, the authors take advantage of the observation that the pressure is essentially a function of the two-dimensional horizontal variables; as a result, they obtain the $L^6$ estimates on the velocity vector field, which allows them to prove the global well-posedness of strong solutions. The global existence of strong solutions were also obtained later by Kobelkov \cite{KOB06}, see also
the subsequent articles of Kukavica and Ziane \cite{KZ07A,KZ07B} for different boundary condition. In all the papers \cite{CAOTITI2,KOB06,KZ07A,KZ07B}, system are assumed to have diffusion in all directions. Recently, it is whown by Cao and Titi \cite{CAOTITI3} that these global existence results still hold true for system with only vertical diffusion, provided the local in time strong solutions exist.

The aims of this paper are two folds: on one hand, we establish the local existence of strong solutions to the primitive equations with only vertical diffusion, provided the initial data belong to $H^2$; on the other hand, we prove that this local strong solution can be in fact extended to be a global one by adopting the energy estimates established in \cite{CAOTITI3} and
some suitable $t$-weighted estimates, as well as our local existence result. Note that the regularity assumptions on the initial data in this paper are weaker than those in Cao and Titi \cite{CAOTITI3} and consequently we improve the results of \cite{CAOTITI3}.

In this paper, we consider the problem in the domain $\Omega_0=M\times(-h,0)$ with $M=(0,1)\times(0,1)$. We complement system (\ref{1.1})--(\ref{1.4}) with the boundary conditions
\begin{eqnarray}
& v, w, T \mbox{ are }\mbox{periodic in }x \mbox{ and }y ,\label{1.5}\\
&(\partial_zv,w)|_{z=-h,0}=0,\label{1.7}\\
&T|_{z=-h}=1,\quad T|_{z=0}=0,\label{1.8}
\end{eqnarray}
and the initial data
\begin{eqnarray}
&(v,T)|_{t=0}=(v_0, T_0). \label{1.9}
\end{eqnarray}

Replacing $T$ and $p$ by $T+\frac{z}{h}$ and $p-\frac{z^2}{2h}$, respectively, then system (\ref{1.1})--(\ref{1.4}) with (\ref{1.5})--(\ref{1.9}) is equivalent to the following system
\begin{eqnarray}
&\partial_tv+(v\cdot\nabla_H)v+w\partial_zv+\nabla_Hp+L_1v+f_0k\times v=0,\label{1-1.1}\\
&\partial_zp+T=0,\label{1-1.2}\\
&\nabla_H\cdot v+\partial_zw=0,\label{1-1.3}\\
&\partial_tT+(v\cdot\nabla_H)T+w\left(\partial_zT+\frac{1}{h}\right)+L_2T=0,\label{1-1.4}
\end{eqnarray}
complemented with the boundary and initial conditions
\begin{eqnarray}
& v, w, T \mbox{ are }\mbox{periodic in }x \mbox{ and }y ,\label{1-1.5}\\
&(\partial_zv,w)|_{z=-h,0}=0,\quad T|_{z=-h,0}=0,\label{1-1.8}\\
&(v,T)|_{t=0}=(v_0, T_0). \label{1-1.9}
\end{eqnarray}
Here, for simplicity, we still use $T_0$ to denote the initial temperature in (\ref{1-1.9}), though it is now different from that in (\ref{1.9}).

Notice that the periodic subspace $\mathcal H$, given by
\begin{align*}
  \mathcal H:=&\{(v,w,p,T)|v,w,p\mbox{ and }T\mbox{are spatially periodic in all three variables} \\
  &\mbox{and even, odd, even and odd in }z \mbox{ variable},\mbox{ respectively}\},
\end{align*}
is invariant under the dynamics system (\ref{1.1})--(\ref{1.4}). That is if the initial data satisfy the properties stated in the definition of $\mathcal H$, then, as we will see later (see Theorem \ref{thm1}), the solutions to system (\ref{1.1})--(\ref{1.4}) will obey the same symmetry as the initial data. This motivated us to consider the following system
\begin{eqnarray}
&\partial_tv+(v\cdot\nabla_H)v+w\partial_zv+\nabla_Hp+L_1v+f_0k\times v=0,\label{1.1-1}\\
&\partial_zp+T=0,\label{1.2-1}\\
&\nabla_H\cdot v+\partial_zw=0,\label{1.3-1}\\
&\partial_tT+(v\cdot\nabla_H)T+w\left(\partial_zT+\frac{1}{h}\right)+L_2T=0,\label{1.4-1}
\end{eqnarray}
in $\Omega:=M\times(-h,h)$, subject to the boundary and initial conditions
\begin{eqnarray}
& v, w, p \mbox{ and } T \mbox{ are }\mbox{periodic in }x, y, z,\label{1.5-1}\\
& v\mbox{ and }p \mbox{ are even in }z,\mbox{ and } w\mbox{ and }T\mbox{ are odd in }z,\label{1.5-3}\\
&(v,T)|_{t=0}=(v_0, T_0). \label{1.9-1}
\end{eqnarray}
One can easily check that the restriction on the sub-domain $\Omega_0$ of a solution $(v,w,p,T)$ to system (\ref{1.1-1})--(\ref{1.9-1}) is a solution to the original system (\ref{1-1.1})--(\ref{1-1.9}).
Because of this, throughout this paper, we mainly concern on the study of system (\ref{1.1-1})--(\ref{1.9-1}) defined on $\Omega$, while the well-posedness results for system (\ref{1-1.1})--(\ref{1-1.9}) defined on $\Omega_0$ follow as a corollary of those for system (\ref{1.1-1})--(\ref{1.9-1}).

For any function $\phi(x,y,z)$ defined on $\Omega$, we denote
$$
\bar\phi(x,y)=\frac{1}{2h}\int_{-h}^h\phi(x,y,z)dz,\qquad\tilde\phi=\phi-\bar\phi.
$$
System (\ref{1.1-1})--(\ref{1.9-1})
is equivalent to (see, e.g., \cite{CAOTITI3})
\begin{eqnarray}
&\partial_tv+L_1v+(v\cdot\nabla_H)v-\left(\int_{-h}^z\nabla_H\cdot v(x,y,\xi,t)d\xi\right)\partial_zv\nonumber\\
&\quad\qquad+f_0k\times v+\nabla_H\left(p_s(x,y,t)-\int_{-h}^zT(x,y,\xi,t)d\xi\right)=0,\label{1.10}\\
&\nabla_H\cdot\bar v=0,\label{1.11}\\
&\partial_tT+L_2T+v\cdot\nabla_HT-\left(\int_{-h}^z\nabla_H\cdot v(x,y,\xi,t)d\xi\right)\left(\partial_zT+\frac{1}{h}\right)=0\label{1.12}
\end{eqnarray}
in $\Omega=M\times(-h,h)$, complemented with the following boundary and initial conditions
\begin{eqnarray}
&v\mbox{ and } T \mbox{ are periodic in }x, y, z,\label{1.13}\\
&v\mbox{ and } T \mbox{ are even and odd in }z,\mbox{ respectively},\\
&(v,T)|_{t=0}=(v_0, T_0). \label{1.16}
\end{eqnarray}
In addition, one can also check that $\bar v$ and $\tilde v$ satisfy the following system (see, e.g., \cite{CAOTITI3})
\begin{eqnarray}
&\partial_t\bar v-\frac{1}{R_1}\Delta_H\bar v+(\bar v\cdot\nabla_H)\bar v+\overline{(\tilde v\cdot\nabla_H)\tilde v+(\nabla_H\cdot\tilde v)\tilde v}+f_0k\times\bar v\nonumber\\
&+\nabla_H\left(p_s(x,y,t)-\frac{1}{2h}\int_{-h}^h\int_{-h}^zT(x,y,\xi,t)d\xi dz\right)=0,\label{1.17}\\
&\nabla_H\cdot\bar v=0,\label{1.18}\\
&\partial_t\tilde v+L_1\tilde v+(\tilde v\cdot\nabla_H)\tilde v-\left(\int_{-h}^z\nabla_H\cdot\tilde v(x,y,\xi,t)d\xi\right)\partial_z\tilde v+(\tilde v\cdot\nabla_H)\bar v\nonumber\\
&+(\bar v\cdot\nabla_H)\tilde v-\overline{(\tilde v\cdot\nabla_H)\tilde v+(\nabla_H\cdot\tilde v)\tilde v}+f_0k\times\tilde v\nonumber\\
&-\nabla_H\left(\int_{-h}^zT(x,y,\xi,t)d\xi-\frac{1}{2h}\int_{-h}^h\int_{-h}^zT(x,y,\xi,t)d\xi dz\right)=0.\label{1.19}
\end{eqnarray}

Throughout this paper, we denote by $L^q(\Omega), L^q(M)$ and $W^{m,q}(\Omega), W^{m,q}(M)$ the standard Lebesgue and Sobolev spaces, respectively. For $q=2$, we use $H^m$ instead of $W^{m,2}$. We use $W_{\text{per}}^{m,q}(\Omega)$ and $H^m_{\text{per}}$ to denote the spaces of periodic functions in $W^{m,q}(\Omega)$ and $H^m(\Omega)$, respectively.
For simplicity, we use the same notations $L^p$ and $H^m$ to denote the $N$ product spaces $(L^p)^N$ and $(H^m)^N$, respectively. We always use $\|u\|_p$ to denote the $L^p$ norm of $u$.

Definitions of the strong solution, maximal existence time and global strong solution are stated in the following three definitions, respectively.

\begin{definition}\label{def1.1}
Let $v_0\in H^2(\Omega)$ and $T_0\in H^2(\Omega)$ be two periodic functions, such that they are even and odd in $z$, respectively. Given a positive number $t_0$. A couple $(v,T)$ is called a strong solution to system
(\ref{1.10})--(\ref{1.16}) (or equivalently (\ref{1.1-1})-(\ref{1.9-1})) on $\Omega\times(0,t_0)$ if

(i) $v$ and $T$ are periodic in $x,y,z$, and they are even and odd in $z$, respectively;

(ii) $v$ and $T$ have the regularities
\begin{align*}
&v\in L^\infty(0,t_0; H^2(\Omega))\cap C([0,t_0]; H^1(\Omega))\cap L^2(0,t_0; H^3(\Omega))\\
&T\in L^\infty(0,t_0; H^2(\Omega))\cap C([0,t_0]; H^1(\Omega)),\quad\partial_zT\in L^2(0,t_0; H^2(\Omega)),\\ &\partial_tv\in L^2(0,t_0; H^1(\Omega)),\quad\partial_tT\in L^2(0,t_0;H^1(\Omega));
\end{align*}

(iii) $v$ and $T$ satisfies (\ref{1.10})--(\ref{1.12}) a.e. in $\Omega\times(0,t_0)$ and the initial condition (\ref{1.16}).
\end{definition}

\begin{definition}
A finite positive number $T^*$ is called the maximal existence time of a strong solution $(v,T)$ to system (\ref{1.10})--(\ref{1.16}) if $(v,T)$ is a strong solution to system on
$\Omega\times(0, t_0)$ for any $t_0<T^*$ and $
\displaystyle\varlimsup_{t\rightarrow T^*_-}(\|v\|_{H^2}^2+\|T\|_{H^2}^2)=\infty.
$
\end{definition}

\begin{definition}
A couple $(v,T)$ is called a global strong solution to system (\ref{1.10})--(\ref{1.16}) if it is a strong solution on $\Omega\times(0,t_0)$ for any $t_0<\infty$.
\end{definition}

The main result of this paper is the following:

\begin{theorem}\label{thm1}
Suppose that the periodic functions $v_0,T_0\in H^2(\Omega)$ are even and odd in $z$, respectively. Then system (\ref{1.1-1})-(\ref{1.9-1}) has a unique global strong solution $(v,T)$.
\end{theorem}

As a first step of proving Theorem \ref{thm1}, we prove the local existence of strong solutions. This is done by regularizing the original system, solving the regularized
system and then taking the limit as the regularization parameter $\varepsilon$ tends to zero. More precisely, we first prove the local existence of strong solutions to the regularized system by the contraction mapping principle, then we prove
that the existence time and the corresponding a priori estimates for these solutions are
independent of the regularization parameter $\varepsilon$, and finally, thanks to these uniform
estimates, we can take the limit to obtain the local strong solutions the original system. By adopting the energy inequalities established in \cite{CAOTITI3} and doing the $t$-weighted estimates on the high order derivatives, we prove that the $H^2$ norms of the solutions keep finite for any finite time, and thus prove the global existence of strong solutions.

As a corollary of Theorem \ref{thm1}, we have the following theorem, which states the well-posedness of strong solutions to system (\ref{1-1.1})--(\ref{1-1.9}). The strong solutions to system (\ref{1-1.1})--(\ref{1-1.9}) are defined in the similar way as before.

\begin{theorem}\label{thm2}
Let $v_0$ and $T_0$ be two functions such that they are periodic in $x$ and $y$. Denote by $v_0^{ext}$ and $T_0^{ext}$ the even and odd extensions in $z$ of $v_0$ and $T_0$, respectively. Suppose that $v_0^{ext}, T_0^{ext}\in H^2_{\text{per}}(\Omega)$. Then system (\ref{1-1.1})--(\ref{1-1.9}) has a unique global strong solution $(v,T)$.
\end{theorem}

The existence part follows directly by applying Theorem \ref{thm1} with initial data $(v_0^{\text{ext}}, T_0^{\text{ext}})$ and restricting the solution on the sub-domain $\Omega_0$. While the uniqueness part can be proven in the same way as that for Theorem \ref{thm1}.

\begin{remark}
  The condition that $v_0^{ext}, T_0^{ext}\in H^2_{\text{per}}(\Omega)$ in the above theorem is necessary for the existence of strong solutions to system (\ref{1-1.1})--(\ref{1-1.9}). Using the similar arguments as stated in the appendix section of this paper, one can show that strong solutions are in fact smooth away from the initial time. It follows from equation (\ref{1-1.4}) and the boundary condition (\ref{1-1.8}) that $\partial_z^2T|_{z=-h,0}=0$ for any $t>0$, and thus we can extend $T$ oddly and periodically in $z$ such that $T^{\text{ext}}$ is odd in $z$ and belongs to $H_{\text{per}}^2(\Omega)$ for any $t>0$. By the definition of strong solutions, it follows that $T\in L^\infty(0,t_0; H^2(\Omega_0))\cap C([0,t_0]; L^2(\Omega_0))$, and thus $T^{\text{ext}}\in L^\infty(0,t_0; H^2(\Omega))\cap C([0,t_0]; L^2(\Omega))$. Combining these statements, by Banach-Alaoglu theorem, it must have $T_0^{\text{ext}}\in H_{\text{per}}^2$. Similarly, one can verify that $v_0^{\text{ext}}$ must belong to $H_{\text{per}}^2(\Omega)$.
\end{remark}

The rest of this paper is arranged as follows: in the next section, section \ref{sec2}, we prove the local existence of strong solutions to the regularized system; in section \ref{sec3}, we establish the local existence and uniqueness of strong solutions to system (\ref{1.1-1})--(\ref{1.9-1}); in section \ref{sec4}, we show that the local strong solution can be extended to be a global one and thus obtain a global strong solution; some necessary regularities used in section \ref{sec4} are justified in the appendix section.

Throughout this paper, the constant $C$ denotes a general constant which may be different from line to line.

\section{The regularized system with full Diffusion}\label{sec2}

In this section, we prove the local existence of strong solutions to the following
modified system
\begin{eqnarray}
&\partial_tv+L_1v+(v\cdot\nabla_H)v-\left(\int_{-h}^z\nabla_H\cdot v(x,y,\xi,t)d\xi\right)\partial_zv\nonumber\\
&\quad\qquad+f_0k\times v+\nabla_H\left(p_s(x,y,t)-\int_{-h}^zT(x,y,\xi,t)d\xi\right)=0,\label{2.0-1}\\
&\nabla_H\cdot\bar v=0,\label{2.0-2}\\
&\partial_tT+L_2T-\varepsilon\Delta_HT+v\cdot\nabla_HT-\left(\int_{-h}^z\nabla_H\cdot v(x,y,\xi,t)d\xi\right)\left(\partial_zT+\frac{1}{h}\right)=0,\label{2.0-3}
\end{eqnarray}
complemented with the boundary and initial conditions
\begin{eqnarray}
&v\mbox{ and }T \mbox{ are periodic in }x, y, z,\label{2.0-6}\\
&v\mbox{ and } T \mbox{ are even and odd in }z,\mbox{ respectively},\\
&(v,T)|_{t=0}=(v_0, T_0). \label{2.0-7}
\end{eqnarray}

Strong solutions to system (\ref{2.0-1})--(\ref{2.0-7}) are defined in the similar way as Definition \ref{def1.1}. We have the following proposition.

\begin{proposition}\label{prop1}
Given $\varepsilon>0$. Let $v_0$ and $T_0\in H^2(\Omega)$ be two periodic functions, such that they are even and odd in $z$, respectively. Then system (\ref{2.0-1})--(\ref{2.0-7}) has a strong solution $(v,T)$ on $\Omega\times(0,t_\varepsilon)$ such that
$$
(v,T)\in L^2(0,t_\varepsilon; H^3(\Omega)),\qquad(\partial_tv,\partial_tT)\in L^2(0,t_\varepsilon; H^1(\Omega)),
$$
where $t_\varepsilon>0$ depends only on $R_1,R_2,R_3,h,\varepsilon$ and the initial data.
\end{proposition}

We will use the contractive mapping principle to prove this proposition. We first introduce the
function spaces and define the mapping. For any given positive number $t_0$, we define the spaces
\begin{align*}
X_0=&\big\{\phi~|~\phi\in C([0,t_0];H^2(M))\cap L^2(0,t_0; H^3(M)), \phi \mbox{ is periodic}\big\},\\
X=&\big\{v~|~v\in C([0,t_0];H^2(\Omega))\cap L^2(0,t_0; H^3(\Omega)), \nabla_H\cdot\bar v=0, \\
& v\mbox{ is periodic in }x,y,z\mbox{ and even in }z\big\},\\
Y=&\big\{T~|~T\in C([0,t_0];H^2(\Omega))\cap L^2(0,t_0; H^3(\Omega)), \\
&T\mbox{ is periodic in }x,y,z \mbox{ and odd in }z\big\},
\end{align*}
and set $\mathcal M_{t_0}=X\times Y$. The norms of these function spaces are defined in the
natural way.

For any given $(v,T)\in\mathcal M_{t_0}$, define a map $ \mathfrak{F}:\mathcal M_{t_0}\rightarrow \mathcal M_{t_0}$ as follows
\begin{equation}
\mathfrak{F}(v,T)=(\mathcal V,\mathcal T),\quad\mathcal V=U+V,\label{2.1}
\end{equation}
where $(U,V,\mathcal T)$ is the unique solution to
\begin{align}
&\partial_t U-\frac{1}{R_1}\Delta_HU+\nabla_Hp=A(v,T),&\mbox{in }M\times(0,t_0),\label{2.2}\\
&\nabla_H\cdot U=0,&\mbox{in }M\times(0,t_0),\label{2.3}\\
&\partial_tV+L_1V=B(v,T),&\mbox{in }\Omega\times(0,t_0),\label{2.4}\\
&\partial_t\mathcal T-\varepsilon\Delta_H\mathcal T+L_2\mathcal T=E(v,T),&\mbox{in }\Omega\times(0,t_0),\label{2.5}
\end{align}
with boundary and initial conditions
\begin{eqnarray}
&U\mbox{ is periodic in }x\mbox{ and }y,\label{2.6}\\
&V\mbox{ and }\mathcal T\mbox{ are periodic in }x,y,z,\label{2.8}\\
&V\mbox{ and }\mathcal T\mbox{ are even and odd in }z, \mbox{ respectively,}\\
&(U,V,\mathcal T)|_{t=0}=(\bar v_0, \tilde v_0, T_0). \label{2.9}
\end{eqnarray}
Here the nonlinear operators $A(v,T), B(v,T)$ and $E(v,T)$ in (\ref{2.2})--(\ref{2.5}) are given by
\begin{align}
A(v,T)=&-(\bar v\cdot\nabla_H)\bar v-\overline{(\tilde v\cdot\nabla_H)\tilde v+(\nabla_H\cdot\tilde v)\tilde v}-f_0k\times\bar v\nonumber\\
&+\nabla_H\left(\frac{1}{2h}\int_{-h}^h\int_{-h}^zT(x,y,\xi,t)d\xi dz\right),\label{2.10}\\
B(v,T)=&-(\tilde v\cdot\nabla_H)\tilde v+\left(\int_{-h}^z\nabla_H\cdot\tilde v(x,y,\xi,t)d\xi\right)\partial_z\tilde v-(\tilde v\cdot\nabla_H)\bar v\nonumber\\
&-(\bar v\cdot\nabla_H)\tilde v+\overline{(\tilde v\cdot\nabla_H)\tilde v+(\nabla_H\cdot\tilde v)\tilde v}-f_0k\times\tilde v\nonumber\\
&+\nabla_H\left(\int_{-h}^zT(x,y,\xi,t)d\xi-\frac{1}{2h}\int_{-h}^h\int_{-h}^zT(x,y,\xi,t)d\xi dz\right),\label{2.11}\\
E(v,T)=&-v\cdot\nabla_HT+\left(\int_{-h}^z\nabla_H\cdot v(x,y,\xi,t)d\xi\right)\left(\partial_zT
+\frac{1}{h}\right).\label{2.12}
\end{align}

For any $(v,T)\in\mathcal M_{t_0}$, one can check that $A(v,T)$ is periodic in $x,y$, $B(v,T)$ is periodic in $x,y,z$ and even in $z$, and $E(v,T)$ is periodic in $x,y,z$ and odd in $z$. In addition, one has $\overline{B(v,T)}=0$, and thus by equation (\ref{2.4}), $\overline V$ satisfies
\begin{equation*}
\partial_t\overline V-\frac{1}{R_1}\Delta_H\overline V=0,\quad\mbox{in }M\times(0,t_0).
\end{equation*}
This implies that $\overline V\equiv0$.
One can easily check that, for any given $(v,T)\in\mathcal M_{t_0}$
\begin{eqnarray*}
&A(v,T)\in L^2(0,t_0; H^1(M)),\nonumber\\
&B(v,T), E(v,T)\in L^2(0,t_0; H^1(\Omega)).
\end{eqnarray*}
By standard $L^2$ theory of linear Stokes equations and parabolic equations, there is a unique solution $(U,V,\mathcal T)\in X_0\times X\times Y$ to system (\ref{2.2})--(\ref{2.9}), such that
$$
\partial_tU\in L^2(0,t_0; H^1(M)), \quad\partial_t V\in L^2(0,t_0; H^1(\Omega)),\quad\partial_t\mathcal T\in L^2(0,t_0; H^1(\Omega)).
$$
Recalling that $\nabla_H\cdot U=0$ and $\overline V\equiv0$, it follows that $\nabla_H\cdot\overline{\mathcal V}=\nabla_H\cdot U+\nabla_H\overline V=0$. Combining these statements, the mapping $\mathfrak F$, given by (\ref{2.1}),
is well defined, and it has an extra regularity
\begin{equation}\label{nw3}
\partial_t\mathfrak F(V,T)\in L^2(0,t_0; H^1(\Omega)).
\end{equation}

One can easily check that
\begin{align*}
A(v,T)+B(v,T)=&-(v\cdot\nabla_H)v+\left(\int_{-h}^z\nabla_H\cdot v(x,y,\xi,t)d\xi\right)\partial_zv\\
&+\nabla_H\left(\int_{-h}^z
T(x,y,\xi,t)d\xi\right)-f_0k\times v=:D(v,T),
\end{align*}
and that
$$
\overline{D(v,T)}=A(v,T),\qquad B(v,T)=D(v,T)-\overline{D(v,T)}.
$$
As a result, recalling (\ref{2.2})--(\ref{2.9}) and $\overline V\equiv0$, $(\mathcal V, \mathcal T)$ satisfies
\begin{eqnarray*}
&\partial_t\mathcal V+L_1\mathcal V+\nabla_H p(x,y,t)=D(v,T),\\
&\nabla_H\cdot \overline{\mathcal V}=0,\\
&\partial_t\mathcal T-\varepsilon\Delta_H\mathcal T+L_2\mathcal T=E(v,T),
\end{eqnarray*}
subject to the boundary and initial value conditions
\begin{eqnarray*}
&\mathcal V\mbox{ and }\mathcal T\mbox{ are periodic in }x,y,z,\\
&\mathcal V\mbox{ and }\mathcal T\mbox{ are even and odd in }z, \mbox{ respectively},\\
&(\mathcal V,\mathcal T)|_{t=0}=(v_0, T_0).
\end{eqnarray*}
Therefore, to find a strong solution to system (\ref{2.0-1})--(\ref{2.0-7}), it suffices to find a fixed point of the mapping $\mathfrak F$ in $\mathcal M_{t_0}$.

Before continuing our arguments, let's state and prove the following lemma on differentiation under the integral sign and integration by parts.

\begin{lemma}\label{lem2.0}
Let $f$ and $g$ be two spatial periodic functions such that
\begin{eqnarray*}
&f\in L^2(0,t_0; H^3(\Omega)),\quad\partial_tf\in L^2(0,t_0; H^1(\Omega)),\\
&g\in L^2(0,t_0; H^2(\Omega)),\quad\partial_tg\in L^2(0,t_0; L^2(\Omega)).
\end{eqnarray*}
Then it follows that
\begin{eqnarray*}
&&\frac{d}{dt}\int_\Omega|\Delta f|^2dxdydz=-2\int_\Omega\nabla\partial_tf\cdot\nabla\Delta fdxdydz,\\
&&\int_\Omega\nabla\partial_{x^i}^2f\cdot\nabla\Delta fdxdydz=\int_\Omega|\partial_{x^i}\Delta f|^2dxdydz
\end{eqnarray*}
and
\begin{eqnarray*}
&&\frac{d}{dt}\int_\Omega|\partial_{x^i} g|^2dxdydz=-2\int_\Omega\partial_tg\partial_{x^i}^2 gdxdydz,\\
&&\int_\Omega\partial_{x^i}^2g\partial_{x^j}^2 gdxdydz=\int_\Omega|\partial_{x^i}\partial_{x^j}g|^2dxdydz
\end{eqnarray*}
for a.e. $t\in(0,t_0)$, where $x^i, x^j\in\{x,y,z\}$.
\end{lemma}

\begin{proof}
The idea of proof follows in the similar lines like the proof of a lemma of Lions (see, e.g., Lemma 1.2 in page 260 of Temam \cite{TEMAM}). We only prove the identities concerning $f$, those for $g$ can be done in the same way. By standard regularization, one can easily show that there is a sequence of smooth functions $\{f_n\}$, such that $f_n$ is periodic in space variables and
$$
f_n\rightarrow f\quad\mbox{ in }L^2(0,t_0; H^3(\Omega)),\qquad \partial_t f_n\rightarrow\partial_t f\quad\mbox{ in }L^2(0,t_0; L^2(\Omega)).
$$
Take arbitrary function $\varphi(t)\in C_0^\infty((0,t_0))$. It is obviously that
\begin{align*}
&\int_0^{t_0}\varphi'(t)\left(\int_\Omega|\Delta f|^2dxdydz\right)dt\\
=&\lim_{n\rightarrow\infty}\int_0^{t_0}\varphi'(t)\left(\int_\Omega|\Delta f_n|^2dxdydz\right)dt\\
=&-2\lim_{n\rightarrow\infty}\int_0^{t_0}\varphi(t)\left(\int_\Omega\Delta f_n\Delta\partial_tf_ndxdydz\right)dt\\
=&2\lim_{n\rightarrow\infty}\int_0^{t_0}\varphi(t)\left(\int_\Omega\nabla\partial_tf_n\nabla\Delta f_ndxdydz\right)dt\\
=&2\int_0^{t_0}\varphi(t)\left(\int_\Omega\nabla\partial_tf\nabla\Delta fdxdydz\right)dt,
\end{align*}
and
\begin{align*}
&\int_0^{t_0}\varphi(t)\left(\int_\Omega\nabla\partial_{x^i}^2f\nabla\Delta fdxdydz\right)dt\\
=&\lim_{n\rightarrow\infty}\int_0^{t_0}\varphi(t)\left(\int_\Omega\nabla\partial_{x^i}^2f_n\nabla\Delta f_ndxdydz\right)dt\\
=&\lim_{n\rightarrow\infty}\int_0^{t_0}\varphi(t)\left(\int_\Omega|\partial_{x^i}\Delta f_n|^2dxdydz\right)dt\\
=&\int_0^{t_0}\varphi(t)\left(\int_\Omega|\partial_{x^i}\Delta f|^2dxdydz\right)dt.
\end{align*}
These identities imply the conclusion.
\end{proof}

\begin{proposition}\label{lem2.1}
Given arbitrary positive number $K$ and time $t_0$, such that $K\geq1$ and $0<t_0\leq 1$, and set $\mathbb B_K=\{(v,T)\in\mathcal M_{t_0}\big|\|(v,T)\|_{\mathcal M_{t_0}}\leq K\}$. Then for any $(v_1,T_1), (v_2,T_2)\in\mathbb B_K$, we have
$$
\|\mathfrak F(v_1,T_1)-\mathfrak F(v_2,T_2)\|_{\mathcal M_{t_0}}\leq C_\varepsilon K t_0^{1/4}\|(v_1-v_2,T_1-T_2)\|_{\mathcal M_{t_0}},
$$
where $C_\varepsilon$ is a constant depending only on $R_1, R_2, R_3, h$ and $\varepsilon$.
\end{proposition}

\begin{proof}
It follows from the H\"older and the Sobolev embedding inequalities that
\begin{align}
&\int_\Omega(|D(v_1,T_1)-D(v_2,T_2)|^2+|\nabla(D(v_1,T_1)-D(v_2,T_2))|^2)dxdydz\nonumber\\
\leq&C\int_\Omega\bigg[|v_1|^2|\nabla(v_1-v_2)|^2+|v_1-v_2|^2|\nabla v_2|^2+
|v_1|^2|\nabla^2(v_1-v_2)|^2\nonumber\\
&+|v_1-v_2|^2|\nabla^2 v_2|^2+|\nabla(v_1-v_2)|^2(|\nabla v_1|^2+|\nabla v_2|^2)\nonumber\\
&+\left(\int_{-h}^h|\nabla(T_1-T_2)|d\xi\right)^2+\left(\int_{-h}^h|\nabla^2(T_1-T_2)|
d\xi\right)^2\nonumber\\
&+\left(\int_{-h}^h|\nabla v_1|d\xi\right)^2|
\nabla(v_1-v_2)|^2+\left(\int_{-h}^h|\nabla(v_1-v_2)|d\xi\right)^2|\nabla v_2|^2
\nonumber\\
&+\left(\int_{-h}^h|\nabla v_1|d\xi\right)^2|
\nabla^2(v_1-v_2)|^2+\left(\int_{-h}^h|\nabla(v_1-v_2)|d\xi\right)^2|\nabla^2 v_2|^2
\nonumber\\
&+\left(\int_{-h}^h|\nabla^2 v_1|d\xi\right)^2|
\nabla(v_1-v_2)|^2+\left(\int_{-h}^h|\nabla^2(v_1-v_2)|d\xi\right)^2|\nabla v_2|^2\bigg]dxdydz\nonumber\\
\leq&C[\|v_1\|_\infty\|\nabla(v_1-v_2)\|_2^2+\|v_1-v_2\|_\infty^2\|\nabla v_2\|_2^2+\|v_1\|_\infty^2\|\nabla^2(v_1-v_2)\|_2^2\nonumber\\
&+\|v_1-v_2\|_\infty^2\|\nabla^2v_2\|_2^2+\|\nabla(v_1-v_2)\|_4^2(\|\nabla v_1\|_4+\|\nabla v_2\|_4^2)+\|\nabla(T_1-T_2)\|_{H^1}^2\nonumber\\
&+\|\nabla v_1\|_6^2\|\nabla(v_1-v_2)\|_2\|\nabla(v_1-v_2)\|_6+\|\nabla(v_1-v_2)\|_6^2\|\nabla v_2\|_2\|\nabla v_2\|_6\nonumber\\
&+\|\nabla v_1\|_6^2\|\nabla^2(v_1-v_2)\|_2\|\nabla^2(v_1-v_2)\|_6+\|\nabla(v_1-v_2)\|_6^2\|\nabla^2v_2\|_2
\|\nabla^2v_2\|_6\nonumber\\
&+\|\nabla^2v_1\|_2\|\nabla^2v_1\|_6\|\nabla(v_1-v_2)\|_6^2+\|\nabla^2(v_1-v_2)\|_2\|\nabla^2
(v_1-v_2)\|_6\|\nabla v_2\|_6^2\nonumber\\
\leq&C[(\|v_1\|_{H^2}^2+\|v_2\|_{H^2}^2)\|v_1-v_2\|_{H^2}^2+\|v_1\|_{H^2}^2\|v_1-v_2
\|_{H^2}\|v_1-v_2\|_{H^3}\nonumber\\
&+\|v_2\|_{H^2}\|v_2\|_{H^3}\|v_1-v_2\|_{H^2}^2+\|v_1\|_{H^2}\|v_1\|_{H^3}\|v_1-v_2\|_{H^2}^2
\nonumber\\
&+\|v_2\|_{H^2}^2\|v_1-v_2\|_{H^2}\|v_1-v_2\|_{H^3}+\|T_1-T_2\|_{H^2}^2]\nonumber\\
\leq&C[(\|v_1\|_{H^2}^2
+\|v_2\|_{H^2}^2)\|v_1-v_2\|_{H^2}\|v_1-v_2\|_{H^3}+\|T_1-T_2\|_{H^2}^2\nonumber\\
&+(\|v_1\|_{H^2}\|v_1\|_{H^3}+\|v_2\|_{H^2}\|v_2\|_{H^3})\|v_1-v_2\|_{H^2}^2]\label{D}
\end{align}
and
\begin{align}
&\int_\Omega(|C(v_1,T_1)-C(v_2,T_2)|^2+|\nabla(C(v_1,T_1)-C(v_2,T_2))|^2)dxdydz\nonumber\\
\leq&C\int_\Omega\bigg[|v_1|^2|\nabla(T_1-T_2)|^2+|v_1-v_2|^2|\nabla T_2|^2+|\nabla v_1|^2|\nabla(T_1-T_2)|^2\nonumber\\
&+|\nabla(v_1-v_2)|^2|\nabla T_2|^2+|v_1|^2|\nabla^2(T_1-T_2)|^2+|v_1-v_2|^2|\nabla^2T_2|^2\nonumber\\
&+\left(\int_{-h}^h|\nabla v_1|d\xi\right)^2|\nabla(T_1-T_2)|^2+\left(\int_{-h}^h|\nabla(v_1-v_2)|d\xi\right)^2(|\nabla T_2|^2+1)\nonumber\\
&+\left(\int_{-h}^h|\nabla^2v_1|d\xi\right)^2|\nabla(T_1-T_2)|^2+\left(\int_{-h}^h|\nabla^2
(v_1-v_2)|d\xi\right)^2(|\nabla T_2|^2+1)\nonumber\\
&+\left(\int_{-h}^h|\nabla v_1|d\xi\right)^2|\nabla^2(T_1-T_2)|^2+\left(\int_{-h}^h|\nabla(v_1-v_2)d\xi\right)^2|\nabla^2T_2
|^2\bigg]dxdydz\nonumber\\
\leq&C\big[\|v_1\|_\infty\|\nabla(T_1-T_2)\|_2^2+\|v_1-v_2\|_\infty^2\|\nabla T_2\|_2^2+\|\nabla
v_1\|_4^2\|\nabla(T_1-T_2)\|_4^2\nonumber\\
&+\|\nabla(v_1-v_2)\|_4^2\|\nabla T_2\|_4^2+\|v_1\|_\infty^2\|\nabla^2(T_1-T_2)\|_2^2+\|v_1-v_2\|_\infty^2\|\nabla^2 T_2\|_2^2\nonumber\\
&+\|\nabla v_1\|_4^2\|\nabla(T_1-T_2)\|_4^2+\|\nabla(v_1-v_2)\|_4^2\|\nabla T_2\|_4^2
+\|\nabla(v_1-v_2)\|_2^2\nonumber\\
&+\|\nabla^2v_1\|_2\|\nabla^2v_1\|_6\|\nabla(T_1-T_2)\|_6^2+\|\nabla^2(v_1-v_2)\|_2^2\nonumber\\
&+\|\nabla^2(v_1-v_2)\|_2\|
\nabla^2(v_1-v_2)\|_6\|\nabla T_2\|_6^2+\|\nabla(v_1-v_2)\|_6^2\|\nabla^2T_2\|_2\|
\nabla^2T_2\|_6\nonumber\\
&+\|\nabla v_1\|_6^2\|\nabla^2(T_1-T_2)\|_2\|\nabla^2(T_1-T_2)\|_6
\big]\nonumber\\
\leq&C(\|v_1\|_{H^2}^2\|T_1-T_2\|_{H^2}^2+\|v_1-v_2\|_{H^2}^2\|T_2\|_{H^2}^2+\|v_1-v_2\|_{H^2}^2
\nonumber\\
&+\|v_1\|_{H^2}\|v_1\|_{H^3}\|T_1-T_2\|_{H^2}^2+\|v_1-v_2\|_{H^2}\|v_1-v_2\|_{H^3}\|T_2\|_{H^2}^2
\nonumber\\
&+\|v_1-v_2\|_{H^2}^2\|T_2\|_{H^2}\|T_2\|_{H^3}+\|v_1\|_{H^2}^2\|T_1
-T_2\|_{H^2}\|T_1-T_2\|_{H^3})
\nonumber\\
\leq&C[(\|v_1\|_{H^2}^2+\|T_2\|_{H^2}^2)(\|v_1-v_2\|_{H^2}\|v_1-v_2\|_{H^3}+\|T_1-T_2\|_{H^2}
\|T_1-T_2\|_{H^3})\nonumber\\
&+(\|v_1\|_{H^2}\|v_1\|_{H^3}+\|T_2\|_{H^2}\|T_2\|_{H^3}+1)(\|v_1-v_2\|_{H^2}^2
+\|T_1-T_2\|_{H^2}^2).\label{C}
\end{align}

Setting $v=v_1-v_2$ and $T=T_1-T_2$, then it follows from (\ref{D}) and (\ref{C}) that
\begin{align*}
I(t)=&\int_\Omega(|D(v_1,T_1)-D(v_2,T_2)|^2+|\nabla(D(v_1,T_1)-D(v_2,T_2))|^2)dxdydz\nonumber\\
&+\int_\Omega(|C(v_1,T_1)-C(v_2,T_2)|^2+|\nabla(C(v_1,T_1)-C(v_2,T_2))|^2)dxdydz\nonumber\\
\leq&C(\|v_1\|_{H^2}\|v_1\|_{H^3}+\|v_2\|_{H^2}\|v_2\|_{H^3}+\|T_2\|_{H^2}\|T_2\|_{H^3}+1)(\|v
\|_{H^2}^2+\|T\|_{H^2}^2)
\nonumber\\
&+C(\|v_1\|_{H^2}^2+\|v_2\|_{H^2}^2+\|T_2\|_{H^2}^2)(\|T\|_{H^2}\|T\|_{H^3}+\|v\|_{H^2}\|v\|_{H^3}).
\end{align*}
Thus, for any $(v_1,T_1),(v_2,T_2)\in\mathbb B_K$, it holds that
\begin{align}
\int_0^{t_0}I(t)dt\leq&CK\|(v,T)\|_{\mathcal M_{t_0}}^2\int_0^{t_0}(\|v_1\|_{H^3}+\|v_2\|_{H^3}+\|T_2\|_{H^3}+1)dt\nonumber\\
&+CK^2\|(v,T)\|_{\mathcal M_{t_0}}\int_0^{t_0}(\|T\|_{H^3}+\|v\|_{H^3})dt\nonumber\\
\leq&CK\|(v,T)\|_{\mathcal M_{t_0}}^2 t_0^{1/2}\left[\int_0^{t_0}(\|v_1\|_{H^3}^2+\|v_2\|_{H^3}^2+\|T_2\|_{H^3}^2+1)dt\right]
^{1/2}\nonumber\\
&+CK^2\|(v,T)\|_{\mathcal M_{t_0}}t_0^{1/2}\left[\int_0^{t_0}(\|T\|_{H^3}^2+\|v\|_{H^3}^2)dt\right]^{1/2}\nonumber\\
\leq&CK^2t_0^{1/2}\|(v,T)\|_{\mathcal M_{t_0}}^2.\label{2.19}
\end{align}

Setting $\mathcal V=\mathcal V_1-\mathcal V_2$ and $\mathcal T=\mathcal T_1-\mathcal T_2$, then $(\mathcal V,\mathcal T)$ satisfies
\begin{eqnarray}
&\partial_t\mathcal V+L_1\mathcal V+\nabla_Hp(x,y,t)=D(v_1,T_1)-D(v_2,T_2),\label{2.22}\\
&\nabla_H\cdot\overline{\mathcal V}=0,\label{2.21}\\
&\partial_t\mathcal T-\varepsilon\Delta_H\mathcal T+L_2\mathcal T=C(v_1,T_1)-C(v_2,T_2)\label{2.23}
\end{eqnarray}
and boundary and initial conditions
\begin{eqnarray*}
&\mathcal V\mbox{ and }\mathcal T\mbox{ are periodic in }x,y,z,\\
&\mathcal V\mbox{ and }\mathcal T\mbox{ are even and odd in }z, \mbox{ respectively},\\
&(\mathcal V,\mathcal T)|_{t=0}=(0, 0).
\end{eqnarray*}

Recalling (\ref{nw3}), it is obviously that $\partial_t\mathcal V, \partial_t\mathcal T\in L^2(0,t_0; H^1(\Omega))$.
Multiplying (\ref{2.22}) and (\ref{2.23}) by $\mathcal V$ and $\mathcal T$, respectively,
and summing the resulting equations up, then it follows from integrating by parts that
\begin{align}
&\frac{1}{2}\frac{d}{dt}\int_\Omega(|\mathcal V|^2+|\mathcal T|^2)dxdydz\nonumber\\
&+\int_\Omega\bigg(\frac{1}{R_1}|\nabla_H\mathcal V|^2+\frac{1}{R_2}|\partial_z\mathcal V|^2+\varepsilon|\nabla_H\mathcal T|^2+\frac{1}{R_3}|\partial_z\mathcal T|^2\bigg)dxdydz\nonumber\\
=&\int_\Omega[(D(v_1,T_1)-D(v_2,T_2))\mathcal V+(C(v_1,T_1)-C(v_2,T_2))\mathcal T]dxdydz.\label{one}
\end{align}
Applying the operator $\nabla$ to equations (\ref{2.22}), (\ref{2.23}) and multiplying the resulting equations by $-\nabla\Delta\mathcal V$ and $-\nabla\Delta\mathcal T$, respectively,
and summing these equations up, then it follows from Lemma \ref{lem2.0} that
\begin{align*}
&\frac{1}{2}\frac{d}{dt}\int_\Omega(|\Delta\mathcal V|^2+|\Delta\mathcal T|^2)dxdydz\\
&+\int_\Omega\bigg(\frac{1}{R_1}|\nabla_H\Delta\mathcal V|^2+\frac{1}{R_2}|\partial_z\Delta\mathcal V|^2+\varepsilon|\nabla_H\Delta\mathcal T|^2+\frac{1}{R_3}|\partial_z\Delta\mathcal T|^2\bigg)dxdydz\\
=&-\int_\Omega[\nabla(D(v_1,T_1)-D(v_2,T_2))\nabla\Delta\mathcal V+\nabla(C(v_1,T_1)-C(v_2,T_2))\nabla\Delta\mathcal T]dxdydz.
\end{align*}
Summing this equation with (\ref{one}) up, and using the Cauchy-Schwarz inequality, one obtains
\begin{align*}
&\frac{1}{2}\frac{d}{dt}\int_\Omega(|\mathcal V|^2+|\Delta \mathcal V|^2+|\mathcal T|^2+|\Delta\mathcal T|^2)dxdudz\nonumber\\
&+\int_\Omega\left[\frac{1}{R_1}(|\nabla_H\mathcal V|^2+|\nabla_H\Delta \mathcal V|^2)+\frac{1}{R_2}(|\partial_z\mathcal V|^2+|\partial_z\Delta\mathcal  V|^2)\right.\nonumber\\
&\left.+\varepsilon(|\nabla_H\mathcal T|^2+|\nabla_H\Delta\mathcal T|^2)+\frac{1}{R_3}(|\partial_z\mathcal T|^2+|\partial_z\Delta\mathcal T|^2)\right]dxdydz\nonumber\\
=&\int_\Omega\big[(D(v_1,T_1)-D(v_2,T_2))\mathcal V-\nabla(D(v_1,T_1)-D(v_2,T_2))\nabla\Delta \mathcal V\nonumber\\
&+(C(v_1,T_1)-C(v_2,T_2))\mathcal T-\nabla(C(v_1,T_1)-C(v_2,T_2))\nabla\Delta\mathcal T\big]dxdydz\nonumber\\
\leq&\sigma\int_\Omega(|\mathcal V|^2+|\nabla\Delta\mathcal  V|^2+|\mathcal T|^2+|\nabla\Delta\mathcal T|^2)dxdydz\nonumber\\
&+C\int_\Omega(|D(v_1,T_1)-D(v_2,T_2)|^2+|\nabla(D(v_1,T_1)
-D(v_2,T_2))|^2)dxdydz\nonumber\\
&+C\int_\Omega(|C(v_1,T_1)-C(v_2,T_2)|^2+|\nabla(C(v_1,T_1)
-C(v_2,T_2))|^2)dxdydz\nonumber\\
\leq&\gamma\int_\Omega(|\mathcal V|^2+|\nabla\Delta\mathcal  V|^2+|\mathcal T|^2+|\nabla\Delta\mathcal T|^2)dxdydz+CI(t),
\end{align*}
where $\gamma=\min\left\{\frac{1}{4R_1},\frac{1}{4R_2}\right\}$.
On account of (\ref{2.19}), it follows from the above inequality that
\begin{align*}
&\sup_{0\leq t\leq t_0}(\|\mathcal V\|_{H^2}^2+\|\mathcal T\|_{H^2}^2)+\int_0^{t_0}
(\|\nabla \mathcal V\|_{H^2}^2+\|\nabla\mathcal T\|_{H^2}^2)dt\nonumber\\
\leq&C\int_0^{t_0}I(t)dt\leq CK^2t_0^{1/2}\|(v,T)\|_{\mathcal M_{t_0}}^2,
\end{align*}
which gives
$$
\|\mathfrak F(v_1,T_1)-\mathfrak F(v_2,T_2)\|_{\mathcal M_{t_0}}^2\leq CK^2t_0^{1/2}\|(v,T)\|_{\mathcal M_{t_0}}^2,
$$
proving the conclusion.
\end{proof}

\begin{proposition}\label{lem2.2}
There is a positive constant $K_0$ depending only on $R_1,R_2,R_3,h,\varepsilon$ and $(v_0,T_0)$, such that
$$
\|\mathfrak F(v,T)\|_{\mathcal M_{t_\varepsilon}}\leq 2K_0,
$$
for any $(v,T)$ with $\|(v,T)\|_{\mathcal M_{t_\varepsilon}}\leq 2K_0$, where $t_\varepsilon=\min\left\{(4C_\varepsilon K_0)^{-4},1\right\}$ and $C_\varepsilon$ is the same constant as in Proposition \ref{lem2.1}.
\end{proposition}

\begin{proof}
Recalling the definition of $\mathfrak F$, the $L^2$ theory of Stokes equations and linear
parabolic equations provide that there is a constant $K_0\geq1$, depending only on $R_1,R_2,R_3,h,\varepsilon$ and $(v_0,T_0)$, such that for any $0<t_0\leq 1$, one has
$$
\|\mathfrak F(0,0)\|_{\mathcal M_{t_0}}\leq K_0.
$$
By the aid of this estimate, applying Proposition \ref{lem2.1}, for any $(v,T)$, with $\|(v,T)\|_{\mathcal M_{t_\varepsilon}}\leq 2K_0$, it holds that
\begin{align*}
\|\mathfrak F(v,T)\|_{\mathcal M_{t_\varepsilon}}\leq&\|\mathfrak F(0,0)\|_{\mathcal M_{t_\varepsilon}}+\|\mathfrak F(v,T)-\mathfrak F(0,0)\|_{\mathcal M_{t_\varepsilon}}\\
\leq&K_0+2C_\varepsilon K_0t_\varepsilon^{1/4}\|(v,T)\|_{M_{t_\varepsilon}}\\
\leq&K_0+4C_\varepsilon K_0^2t_\varepsilon^{1/4}\leq 2K_0,
\end{align*}
provided $t_\varepsilon\leq\min\left\{(4C_\varepsilon K_0)^{-4},1\right\}$.
This completes the
proof.
\end{proof}

Now we are ready to give the proof of Proposition \ref{prop1}.
\begin{proof}[\textbf{Proof of Proposition \ref{prop1}}]
Let $K_0$ be the constant stated in Proposition \ref{lem2.2}. By Lemma
\ref{lem2.1} and Proposition \ref{lem2.2}, the mapping defined by (\ref{2.1})
satisfies
\begin{eqnarray*}
&\mathfrak F: \mathbb B_{2K_0}\rightarrow\mathbb B_{2K_0},\\
&\|\mathfrak F(v_1,T_1)-\mathfrak F(v_2,T_2)\|_{\mathcal M_{t_\varepsilon}}\leq\frac{1}{2}\|(v_1-v_2,T_1-T_2)\|_{\mathcal M_{t_\varepsilon}},
\end{eqnarray*}
for any $(v_1,T_1), (v_2,T_2)\in \mathbb B_{2K_0}$, where
$$
t_\varepsilon=\min\left\{(4C_\varepsilon K_0)^{-4},1\right\},\quad\mathbb B_{2K_0}=\big\{(v,T)\big|\|(v,T)\|_{\mathcal M_{t_\varepsilon}}\leq2K_0\big\}.
$$
By the contraction mapping principle, there is a unique fixed point $(v,T)$ for $\mathfrak F$ in $\mathbb B_{2K_0}$. This fixed point of $\mathfrak F$ is a strong solution to system (\ref{2.0-1})--(\ref{2.0-7}). The regularities $\partial_tv,\partial_tT\in L^2(0,t_\varepsilon; H^1(\Omega))$ follow from (\ref{nw3}), completing the
proof of Proposition \ref{prop1}.
\end{proof}

\section{The system with only vertical diffusion}\label{sec3}

In this section, we prove the local existence and uniqueness of strong solution to system (\ref{1.1-1})--(\ref{1.9-1}), or equivalently system (\ref{1.10})--(\ref{1.16}). The existence is obtained by taking the limit
$\varepsilon\rightarrow0$ of the solutions $(v_\varepsilon, T_\varepsilon)$
to system (\ref{2.0-1})--(\ref{2.0-7}).

We first establish the uniform in $\varepsilon$ lower bounds of the existence times and the estimates on $(v_\varepsilon, T_\varepsilon)$. In fact, we have the following:

\begin{proposition}\label{lem3.1}
There is a positive constant $K_0$ and a positive time $t_0^*$, depending only on $R_1,R_2,R_3,h$ and
$(v_0,T_0)$, such that system (\ref{2.0-1})--(\ref{2.0-7}) has a solution $(v_\varepsilon, T_\varepsilon)$ in $\Omega\times(0,t_0^*)$ with
$$
\sup_{0\leq t\leq t_0^*}(\|v_\varepsilon\|_{H^2}^2+\|T_\varepsilon\|_{H^2}^2)+\int_0^{t_0^*}(\varepsilon\|\nabla_H
T_\varepsilon\|_{H^2}^2+\|\partial_zT_\varepsilon\|_{H^2}^2+\|\nabla v_\varepsilon\|_{H^2}^2)dt\leq K
$$
and
$$
\int_0^{t_0^*}(\|\partial_tv_\varepsilon\|_{H^1}^2+\|\partial_t T_\varepsilon\|_{H^1}^2)dt\leq K,
$$
where $K$ is a positive constant depending only on $R_1,R_2,R_3,h$ and
$(v_0,T_0)$.
\end{proposition}

\begin{proof}
Let $t_\varepsilon^*$ be the maximal existence time of strong solution $(v_\varepsilon, T_\varepsilon)$ to system (\ref{2.0-1})--(\ref{2.0-7}).
Multiplying (\ref{2.0-1}) by $v_\varepsilon$ and integrating over $\Omega$ yields
\begin{align*}
\frac{1}{2}\frac{d}{dt}\int_\Omega|v_\varepsilon|^2dxdydz
&+\int_\Omega
\left(\frac{1}{R_1}|\nabla_Hv_\varepsilon|^2+\frac{1}{R_2}|\partial_zv_\varepsilon|^2\right)dxdydz
\nonumber\\
=&\int_\Omega\nabla_H\left(\int_{-h}^zT_\varepsilon d\xi\right)v_\varepsilon dxdydz.
\end{align*}
Applying the operator $\nabla$ to (\ref{2.0-1}), multiplying the resulting equation by $-\nabla\Delta v_\varepsilon$, summing them up and integrating over $\Omega$, then it follows from Lemma \ref{lem2.0} (recall the regularities of $(v_\varepsilon, T_\varepsilon)$) that
\begin{align*}
&\frac{1}{2}\frac{d}{dt}\int_\Omega|\Delta v_\varepsilon|^2dxdydz+\int_\Omega
\left(\frac{1}{R_1}|\nabla_H\Delta v_\varepsilon|^2+\frac{1}{R_2}|\partial_z\Delta v_\varepsilon|^2\right)dxdydz\nonumber\\
=&\int_\Omega\nabla\left[(
v_\varepsilon\cdot\nabla_H)v_\varepsilon-\left(\int_{-h}^z\nabla_H\cdot v_\varepsilon d\xi\right)
\partial_zv_\varepsilon
-\left(\int_{-h}^z\nabla_HT_\varepsilon  d\xi\right)\right]\nabla\Delta v_\varepsilon dxdydz.
\end{align*}
Summing the above equation with the previous one up, and using
the H\"older, Sobolev, Poincar\'e and Cauchy inequalities, we have
\begin{align}
&\frac{1}{2}\frac{d}{dt}\int_\Omega(|v_\varepsilon|^2+|\Delta v_\varepsilon|^2)dxdydz\nonumber\\
&+\int_\Omega
\left[\frac{1}{R_1}(|\nabla_Hv_\varepsilon|^2+|\nabla_H\Delta v_\varepsilon|^2)\right.+\left.\frac{1}{R_2}(|\partial_zv_\varepsilon|^2+|\partial_z\Delta v_\varepsilon|^2)\right]dxdydz\nonumber\\
=&\int_\Omega\left\{\nabla_H\left(\int_{-h}^zT_\varepsilon d\xi\right)v_\varepsilon+\nabla\left[(
v_\varepsilon\cdot\nabla_H)v_\varepsilon-\left(\int_{-h}^z\nabla_H\cdot v_\varepsilon d\xi\right)
\partial_zv_\varepsilon\right.\right.\nonumber\\
&\left.\left.-\left(\int_{-h}^z\nabla_HT_\varepsilon(x,y,\xi,t) d\xi\right)\right]\nabla\Delta v_\varepsilon\right\}dxdydz\nonumber\\
\leq&C\|\nabla T_\varepsilon\|_2\|v_\varepsilon\|_2+C\int_\Omega\left[|v_\varepsilon||\nabla^2v_\varepsilon|+
|\nabla v_\varepsilon|^2+\left(\int_{-h}^h|\nabla^2v_\varepsilon|d\xi\right)|\partial_zv_\varepsilon|
\right.\nonumber\\
&\left.+\left(\int_{-h}^h|\nabla v_\varepsilon|d\xi\right)|\nabla^2v_\varepsilon|+\left(\int_{-h}^h|\nabla^2 T_\varepsilon|d\xi\right)\right]|\nabla\Delta v_\varepsilon|dxdydz\nonumber\\
\leq&\sigma\int_\Omega|\nabla\Delta v_\varepsilon|^2dxdydz+C\|\nabla T_\varepsilon\|_2\|v_\varepsilon\|_2+C\int_\Omega\Big[|v_\varepsilon|^2|\nabla^2v_\varepsilon|^2
+|\nabla v_\varepsilon|^4\nonumber\\
&+\Big(\int_{-h}^h|\nabla^2v_\varepsilon|dz\Big)^2|\partial_zv_\varepsilon|^2+\Big(\int_{-h}^h|
\nabla v_\varepsilon|dz\Big)^2|\nabla^2v_\varepsilon|^2+\Big(\int_{-h}^h|\nabla^2 T_\varepsilon|dz\Big)^2\Big]dxdydz\nonumber\\
\leq&\sigma\int_\Omega|\nabla\Delta v_\varepsilon|^2dxdydz+C\|\nabla T_\varepsilon\|_2\|v_\varepsilon\|_2+C(\|v_\varepsilon\|_6^2\|\nabla^2v_\varepsilon\|_2
\|\nabla^2v_\varepsilon\|_6\nonumber\\
&+\|\nabla v_\varepsilon\|_4^4+\|\nabla^2v_\varepsilon\|_2\|\nabla^2v_\varepsilon\|_6
\|\partial_zv_\varepsilon\|_6^2
+\|\nabla v_\varepsilon\|_6^2\|\nabla^2v_\varepsilon\|_2\|\nabla^2v_\varepsilon\|_6+\|\nabla^2 T_\varepsilon\|_2^2)\nonumber\\
\leq&\sigma\int_\Omega|\nabla\Delta v_\varepsilon|^2dxdydz+C\|\nabla T_\varepsilon\|_2\|v_\varepsilon\|_2+C(\|v_\varepsilon\|_{H^1}^2\|\nabla^2v_\varepsilon\|_2
\|\nabla\Delta v_\varepsilon\|_2\nonumber\\
&+\|v_\varepsilon\|_{H^2}^4+\|v_\varepsilon\|_{H^2}^3\|\nabla\Delta v_\varepsilon\|_2+\|\nabla^2 T_\varepsilon\|_2^2)\nonumber\\
\leq&2\sigma\int_\Omega|\nabla\Delta v_\varepsilon|^2dxdydz+C(1+\|v_\varepsilon\|_{H^2}^6+\|T_\varepsilon\|_{H^2}^6), \label{3.1}
\end{align}
with $\sigma=\min\left\{\frac{1}{4R_1},\frac{1}{4R_2}\right\}$, and thus we have
\begin{align}
&\sup_{0\leq s\leq t}\|v_\varepsilon\|_{H^2}^2+\int_0^t\|\nabla v_\varepsilon\|_{H^2}^2ds\nonumber\\
\leq&C\|v_0\|_{H^2}^2+C\int_0^t(1+\|T_\varepsilon\|_{H^2}^2+\|v_\varepsilon\|_{H^2}^2)^3ds
\label{3.2}
\end{align}
for any $0\leq t<t_\varepsilon^*$.

Multiplying equation (\ref{2.0-2}) by $T_\varepsilon$, then it follows from integrating by parts that
\begin{align}
&\frac{1}{2}\frac{d}{dt}\int_\Omega |T_\varepsilon|^2dxdydz+\int_\Omega\left(\varepsilon|\nabla_HT_\varepsilon|^2+\frac{1}{R_3}
|\partial_zT_\varepsilon|^2\right)dxdydz\nonumber\\
=&\frac{1}{h}\int_\Omega\left(\int_{-h}^z\nabla_H\cdot v_\varepsilon d\xi\right) T_\varepsilon dxdydz
\leq C\|T_\varepsilon\|_2\|\nabla v_\varepsilon\|_2.\label{3.3}
\end{align}
Recalling the Gagliado-Nirenberg inequality of the form
$$
\|f\|_\infty\leq C\|f\|_6^{1/2}\|f\|_{H^2}^{1/2},\qquad\forall f\in H^2(\Omega), \Omega\subseteq
\mathbb R^3,
$$
it follows from the H\"older, Sobolev and Poincar\'e inequalities that
\begin{align}
\|\partial_zT_\varepsilon\|_\infty\leq& C\|\partial_zT_\varepsilon\|_6^{1/2}(\|\nabla^2\partial_z
T_\varepsilon\|_2^{1/2}+\|\partial_zT_\varepsilon\|_2^{1/2})\nonumber\\
\leq&C\|\Delta T_\varepsilon\|_2^{1/2}(\|\Delta\partial_zT_\varepsilon\|_2^{1/2}+\|\partial_z
T_\varepsilon\|_2^{1/2})\label{A11}
\end{align}
and
\begin{align}
\|\nabla v_\varepsilon\|_\infty\leq&C\|\nabla v_\varepsilon\|_6^{1/2}(\|\nabla^3 v_\varepsilon\|_2^{1/2}+\|\nabla v_\varepsilon\|_2^{1/2})\nonumber\\
\leq&C\|\nabla^2 v_\varepsilon\|_2^{1/2}\|\nabla^3v_\varepsilon\|_2^{1/2}=C
\|\Delta v_\varepsilon\|_2^{1/2}\|\nabla\Delta v_\varepsilon\|_2^{1/2}. \label{A12}
\end{align}
Applying the operator $\nabla$ to equation (\ref{2.0-3}) and multiplying the resulting equation
by $-\nabla\Delta T_\varepsilon$ and integrating over $\Omega$, then it follows from Lemma \ref{lem2.0} (recall the regularities of ($v_\varepsilon,T_\varepsilon)$), (\ref{A11}), (\ref{A12}), the H\"older, Soblolev, Poincar\'e and Cauchy-Schwarz inequalities that
\begin{align}
&\frac{1}{2}\frac{d}{dt}\int_\Omega|\Delta T_\varepsilon|^2dxdydz+\int_\Omega\left(\varepsilon|\nabla_H\Delta T_\varepsilon|^2+\frac{1}{R_3}|\partial_z\Delta T_\varepsilon|^2\right)dxdydz\nonumber\\
=&\int_\Omega\Delta\left[\left(\int_{-h}^z\nabla_H\cdot v_\varepsilon d\xi\right)\left(\partial_zT_\varepsilon+\frac{1}{h}\right)-(v_\varepsilon\cdot\nabla_H)
T_\varepsilon\right]\Delta T_\varepsilon dxdydz\nonumber\\
=&\int_\Omega\left[\left(\int_{-h}^z\Delta\nabla_H\cdot v_\varepsilon d\xi\right)\left(\partial_zT_\varepsilon+\frac{1}{h}\right)+2\left(\int_{-h}^z\nabla\nabla_H\cdot v_\varepsilon d\xi\right)\nabla\partial_zT_\varepsilon\right.\nonumber\\
&-(\Delta v_\varepsilon\cdot\nabla_H)T_\varepsilon-2\nabla v_\varepsilon\cdot\nabla_H\nabla T_\varepsilon\Big]\Delta T_\varepsilon dxdydz\nonumber\\
=&\int_\Omega\left[\left(\int_{-h}^z\Delta\nabla_H\cdot v_\varepsilon d\xi\right)\left(\partial_zT_\varepsilon+\frac{1}{h}\right)-(\Delta v_\varepsilon\cdot\nabla_H)T_\varepsilon-2\nabla v_\varepsilon\cdot\nabla_H\nabla
T_\varepsilon\right]\Delta T_\varepsilon dxdydz\nonumber\\
&-2\int_\Omega\left[\nabla\nabla_H\cdot v_\varepsilon\nabla T_\varepsilon\Delta T_\varepsilon+\left(\int_{-h}^z\nabla\nabla_H\cdot v_\varepsilon d\xi\right)\nabla T_\varepsilon\Delta\partial_zT_\varepsilon\right]dxdydz\nonumber\\
\leq&C(\|\Delta\nabla v_\varepsilon\|_2\|\partial_zT_\varepsilon\|_\infty\|\Delta T_\varepsilon\|_2+\|\Delta\nabla v_\varepsilon\|_2\|\Delta T_\varepsilon\|_2+\|\Delta v_\varepsilon\|_3\|\nabla_HT_\varepsilon\|_6\|\Delta T_\varepsilon\|_2\nonumber\\
&+\|\nabla v_\varepsilon\|_\infty\|\nabla^2T_\varepsilon\|_2^2+\|\nabla^2v_\varepsilon\|_3\|\nabla T_\varepsilon\|_6\|\Delta T_\varepsilon\|_2+\|\nabla^2v_\varepsilon\|_3\|\nabla T_\varepsilon\|_6\|\Delta\partial_zT_\varepsilon\|_2)\nonumber\\
\leq&C\Big[\|\nabla\Delta v_\varepsilon\|_2\|\Delta T_\varepsilon\|_2^{1/2}(\|\Delta\partial_zT_\varepsilon\|_2^{1/2}
+\|\partial_zT_\varepsilon\|_2^{1/2})\|\Delta T_\varepsilon\|_2+\|\Delta\nabla v_\varepsilon\|_2\|\Delta T_\varepsilon\|_2
\nonumber\\
&+\|\Delta v_\varepsilon\|_2^{1/2}\|\nabla\Delta v_\varepsilon\|_2^{1/2}\|\Delta T_\varepsilon\|_2^2+\|\Delta v_\varepsilon\|_2^{1/2}\|\nabla\Delta v_\varepsilon\|_2^{1/2}\|\Delta T_\varepsilon\|_2\|\Delta\partial_z T_\varepsilon\|_2\Big]\nonumber\\
\leq&\sigma(\|\Delta\partial_zT_\varepsilon\|_2^2+\|\nabla\Delta v_\varepsilon\|_2^2)+C_\sigma(1+\|v_\varepsilon\|_{H^2}^6+\|T_\varepsilon\|_{H^2}^6),\label{3.4}
\end{align}
where $\sigma$ is a sufficiently small positive constant.

Combining (\ref{3.3}) with (\ref{3.4}), it follows that
\begin{align*}
&\sup_{0\leq s\leq t}\|T_\varepsilon\|_{H^2}^2+\int_0^t(\varepsilon\|\nabla_HT_\varepsilon\|_{H^2}^2+\|\partial_zT_\varepsilon
\|_{H^2}^2)ds\nonumber\\
\leq&C\|T_0\|_{H^2}^2+\sigma\int_0^t\|\nabla\Delta v_\varepsilon\|_2^2ds+C_\sigma\int_0^t(1+\|T_\varepsilon
\|_{H^2}^2+\|v_\varepsilon\|_{H^2}^2)^3ds,
\end{align*}
with a sufficiently small positive constant $\sigma$, for any $0\leq t<t_\varepsilon^*$. This, combined with (\ref{3.2}), implies
\begin{align*}
&\sup_{0\leq s\leq t}(\|v_\varepsilon\|_{H^2}^2+\|T_\varepsilon\|_{H^2}^2)+\int_0^t(\varepsilon
\|\nabla_HT_\varepsilon\|_{H^2}^2+\|\partial_zT_\varepsilon\|_{H^2}^2+\|\nabla v_\varepsilon\|_{H^2}^2)ds\nonumber\\
\leq&C(\|v_0\|_{H^2}^2+\|T_0\|_{H^2}^2)+C\int_0^t(1+\|T_\varepsilon\|_{H^2}^2+\|v_\varepsilon
\|_{H^2}^2)^3ds
\end{align*}
for any $0\leq t<t_\varepsilon^*$. Set
$$
f(t)=\sup_{0\leq s\leq t}(\|v_\varepsilon\|_{H^2}^2+\|T_\varepsilon\|_{H^2}^2+1)+\int_0^t(\varepsilon
\|\nabla_HT_\varepsilon\|_{H^2}^2+\|\partial_zT_\varepsilon\|_{H^2}^2+\|\nabla v_\varepsilon\|_{H^2}^2)ds
$$
for $t\in[0,t_\varepsilon^*)$. Then one has
\begin{equation}\label{E1}
f(t)\leq CC_0+C\int_0^t(f(s))^3ds,\quad t\in[0,t_\varepsilon^*),
\end{equation}
where $C_0=\|v_0\|_{H^2}^2+\|T_0\|_{H^2}^2$.
Set $F(t)=\int_0^t(f(t)^3)ds+1$, then by (\ref{E1}) one has
$$
F'(t)=(f(t))^3\leq C_1(F(t))^3,\quad\forall t\in[0,t_\varepsilon^*),
$$
where $C_1$ is a positive constant depending only on $R_1,R_2,R_3,h$ and
$(v_0,T_0)$. This inequality implies
$$
F(t)\leq\frac{1}{\sqrt{1-2C_1t}},\quad\forall t\in[0,t_\varepsilon^*)\cap[0,\frac{1}{2C_1}),
$$
and thus
\begin{align*}
&\sup_{0\leq s\leq t}(\|v_\varepsilon\|_{H^2}^2+\|T_\varepsilon\|_{H^2}^2)+\int_0^t(\varepsilon
\|\nabla_HT_\varepsilon\|_{H^2}^2+\|\partial_zT_\varepsilon\|_{H^2}^2+\|\nabla v_\varepsilon\|_{H^2}^2)ds\\
\leq&CC_0+CF(t)\leq CC_0+\frac{C}{\sqrt{1-2C_1t}}\leq C(C_0+\sqrt 2),
\end{align*}
for any $t\in[0,t_\varepsilon^*)\cap[0,\frac{1}{4C_1}].$ Recalling that $t_\varepsilon^*$ is the
maximal existence time, the above inequality implies $t_\varepsilon^*>\frac{1}{4C_1}$, and thus, we can choose $t_0^*=\frac{1}{4C_1}$, and
\begin{align}
\sup_{0\leq s\leq t_0^*}&(\|v_\varepsilon\|_{H^2}^2+\|T_\varepsilon\|_{H^2}^2)+\int_0^{t_0^*}(\varepsilon
\|\nabla_HT_\varepsilon\|_{H^2}^2\nonumber\\
&+\|\partial_zT_\varepsilon\|_{H^2}^2+\|\nabla v_\varepsilon\|_{H^2}^2)ds\leq  C(C_0+\sqrt 2).\label{E2}
\end{align}

We still need to establish estimates on $\partial_tv_\varepsilon$ and $\partial_tT_\varepsilon$. The estimates on $\partial_tT_\varepsilon$ follow from equation (\ref{2.0-3}), the estimates (\ref{E2}) and the Sobolev embedding inequality as follows
\begin{align*}
\int_0^{t_0^*}&\|\partial_tT_\varepsilon\|_{H^1}^2dt\leq C\int_0^{t_0^*}\bigg(\|\partial_z^2T_\varepsilon\|_{H^1}^2+\varepsilon\|\Delta_HT_\varepsilon\|_{H^1}^2
  +\|v_\varepsilon\cdot\nabla_HT_\varepsilon\|_{H^1}^2\\
  &+\left\|\int_{-h}^z\nabla_H\cdot v_\varepsilon d\xi\bigg(\partial_zT_\varepsilon+\frac{1}{h}\bigg)\right\|_{H^1}^2\bigg)dt\\
  \leq&C\int_0^{t_0^*}\bigg(\|\partial_zT_\varepsilon\|_{H^2}^2+\varepsilon\|\nabla_HT_\varepsilon\|
  _{H^2}^2+\|v_\varepsilon\|_{W^{1,\infty}}^2(\|\nabla T_\varepsilon\|_{H^1}^2+1)\\
  &+\left\|\int_{-h}^z\nabla_H\nabla_H\cdot v_\varepsilon d\xi(|\partial_zT_\varepsilon|+1)\right\|_2^2\bigg)dt\\
  \leq&C\int_0^{t_0^*}[\|\partial_zT_\varepsilon\|_{H^2}^2+\varepsilon\|\nabla_HT_\varepsilon\|
  _{H^2}^2+\|v_\varepsilon\|_{W^{1,\infty}}^2(\|\nabla T_\varepsilon\|_{H^1}^2+1)\\
  &+(\|\partial_zT_\varepsilon\|_\infty^2+1)\|\nabla^2v_\varepsilon\|_2^2]dt\\
  \leq&C\int_0^{t_0^*}[\|\partial_zT_\varepsilon\|_{H^2}^2+\varepsilon\|\nabla_HT_\varepsilon\|
  _{H^2}^2+\|v_\varepsilon\|_{H^3}^2(\|T_\varepsilon\|_{H^2}^2+1)+\|\partial_zT_\varepsilon
  \|_{H^2}^2\|v_\varepsilon\|_{H^2}^2]dt\\
  \leq&C(C_0+\sqrt2)^2(t_0^*+1).
  \end{align*}
For the estimates on $\partial_tv_\varepsilon$, we split $v_\varepsilon$ as $v_\varepsilon=\tilde{v}_\varepsilon+\bar v_\varepsilon$ (recall the definitions of $\tilde\phi$ and $\bar\phi$ in the Introduction). Then $\tilde v_\varepsilon$ and $\bar v_\varepsilon$ satisfy system (see e.g. Cao and Titi \cite{CAOTITI2})
\begin{eqnarray}
&\partial_t\bar v_\varepsilon-\frac{1}{R_1}\Delta_H\bar v_\varepsilon+(\bar v_\varepsilon\cdot\nabla_H)\bar v_\varepsilon+\overline{(\tilde v_\varepsilon\cdot\nabla_H)\tilde v_\varepsilon+(\nabla_H\cdot\tilde v_\varepsilon)\tilde v_\varepsilon}+f_0k\times\bar v_\varepsilon\nonumber\\
&+\nabla_H\left(p_s(x,y,t)-\frac{1}{2h}\int_{-h}^h\int_{-h}^zT_\varepsilon(x,y,\xi,t)d\xi dz\right)=0, \nonumber\\
&\nabla_H\cdot\bar v_\varepsilon=0,\nonumber \\
&\partial_t\tilde v_\varepsilon+L_1\tilde v_\varepsilon+(\tilde v_\varepsilon\cdot\nabla_H)\tilde v_\varepsilon-\left(\int_{-h}^z\nabla_H\cdot\tilde v_\varepsilon(x,y,\xi,t)d\xi\right)\partial_z\tilde v_\varepsilon+(\tilde v_\varepsilon\cdot\nabla_H)\bar v_\varepsilon\nonumber\\
&+(\bar v_\varepsilon\cdot\nabla_H)\tilde v_\varepsilon-\overline{(\tilde v_\varepsilon\cdot\nabla_H)\tilde v_\varepsilon+(\nabla_H\cdot\tilde v_\varepsilon)\tilde v_\varepsilon}+f_0k\times\tilde v_\varepsilon\nonumber\\
&-\nabla_H\left(\int_{-h}^zT_\varepsilon(x,y,\xi,t)d\xi-\frac{1}{2h}\int_{-h}^h\int_{-h}^z
T_\varepsilon(x,y,\xi,t)d\xi dz\right)=0.\label{E4}
\end{eqnarray}
Thanks to the estimates (\ref{E2}), we can apply the $L^2$ theory of Stokes equations and the Sobolev embedding inequality to deduce
\begin{align*}
  \int_0^{t_0^*}&\|\partial_t\bar v_\varepsilon\|_{H^1(M)}^2dt\leq C\int_0^{t_0^*}
  \bigg(\|\Delta_H\bar v_\varepsilon\|_{H^1(M)}^2+\|\bar v_\varepsilon\cdot\nabla_H\bar v_\varepsilon\|_{H^1(M)}^2+\|\overline{\tilde v_\varepsilon\cdot\nabla_H\tilde v_\varepsilon}
  \|_{H^1(M)}^2\nonumber\\
  &+\|\overline{\nabla_H\cdot\tilde v_\varepsilon \tilde v_\varepsilon}\|_{H^1(M)}^2+\|\bar v_\varepsilon\|_{H^1(M)}^2+\left\|\nabla_H\int_{-h}^h\int_{-h}^zT_\varepsilon d\xi dz\right\|_{H^1(M)}^2\bigg)dt\nonumber\\
  \leq&C\int_0^{t_0^*}(\|\nabla v_\varepsilon\|_{H^2}^2+\|v_\varepsilon\|_{W^{1,\infty}}^2\|\nabla v_\varepsilon\|_{H^1}^2+
  \|v_\varepsilon\|_{H^2}^2+\|T_\varepsilon\|_{H^2}^2)dt\nonumber\\
  \leq&C\int_0^{t_0^*}(\|\nabla v_\varepsilon\|_{H^2}^2+\|v_\varepsilon\|_{H^3}^2\| v_\varepsilon\|_{H^2}^2+
  \|v_\varepsilon\|_{H^2}^2+\|T_\varepsilon\|_{H^2}^2)dt
  \leq C(C_0+\sqrt2)^2(t_0^*+1),
\end{align*}
and similarily, it follows from equation (\ref{E4}) and the Sobolev embedding inequality that
\begin{align*}
  \int_0^{t_0^*}\|\partial_t\tilde v_\varepsilon\|_{H^1}^2dt\leq C(C_0+\sqrt2)^2(t_0^*+1).
\end{align*}
Therefore, one obtains
$$
\int_0^{t_0^*}\|\partial_tv_\varepsilon\|_{H^1}^2dt\leq C\int_0^{t_0^*}(\|\partial_t\tilde v_\varepsilon\|_{H^1}^2+\|\partial_t\bar v_\varepsilon\|_{H^1(M)}^2)dt\leq C(C_0+\sqrt2)^2(t_0^*+1).
$$
This completes the proof.
\end{proof}

We will use the following lemma to prove the uniqueness.

\begin{lemma}\label{lem3.2}
The following inequalities hold true
\begin{align*}
&\left|\int_M\left(\int_{-h}^hf(x,y,z)dz\right)\left(\int_{-h}^hg(x,y,z)h(x,y,z)dz\right)dxdy\right|\\
\leq&C\|f\|_2^{1/2}\left(\|f\|_2^{1/2}+\|\nabla_Hf\|_{2}^{1/2}\right)\|g\|_2\|h\|_2^{1/2}\left(
\|h\|_2^{1/2}+\|\nabla_Hh\|_2^{1/2}\right),
\end{align*}
and
\begin{align*}
&\left|\int_M\left(\int_{-h}^hf(x,y,z)dz\right)\left(\int_{-h}^hg(x,y,z)h(x,y,z)dz\right)dxdy\right|\\
\leq&C\|f\|_2\|g\|_2^{1/2}\left(\|g\|_2^{1/2}+\|\nabla_Hg\|_{2}^{1/2}\right)\|h\|_2^{1/2}\left(
\|h\|_2^{1/2}+\|\nabla_Hh\|_2^{1/2}\right).
\end{align*}
\end{lemma}

\begin{proof}
The proof can be established in the same way as in Proposition 2.2 of \cite{CAOTITI1}, and thus it is omitted here.
\end{proof}

We also need the following version of the Aubin-Lions lemma.

\begin{lemma}\label{AL}
(Aubin-Lions Lemma, See Simon \cite{Simon} Corollary 4) Assume that $X, B$ and $Y$ are three Banach spaces, with $X\hookrightarrow\hookrightarrow B\hookrightarrow Y.$ Then it holds that

(i) If $F$ is a bounded subset of $L^p(0, T; X)$ where $1\leq p<\infty$, and $\frac{\partial F}{\partial t}=\left\{\frac{\partial f}{\partial t}|f\in F\right\}$ is bounded in $L^1(0, T; Y)$, then $F$ is relatively compact in $L^p(0, T; B)$;

(ii) If $F$ is bounded in $L^\infty(0, T; X)$ and $\frac{\partial F}{\partial t}$ is bounded in $L^r(0, T; Y)$ where $r>1$, then $F$ is relatively compact in $C([0, T]; B)$.
\end{lemma}

\begin{proposition}\label{prop2}
Let $v_0$ and $T_0\in H^2(\Omega)$ be two periodic functions, such that they are even and odd in $z$, respectively. Then system (\ref{1.1-1})-(\ref{1.9-1}) has a unique strong solution $(v,T)$ in $\Omega\times(0,t_0^*)$, such that
$$
(v,T)\in L^\infty(0,t_0^*;H^2(\Omega))\cap C([0,t_0^*];H^1(\Omega)), \quad(\nabla v,\partial_zT)\in L^2(0,t_0^*; H^2(\Omega)
$$
and
$$
(\partial_tv,\partial_tT)\in L^2(0,t_0^*; H^1(\Omega)),
$$
where $t_0^*$ is the same positive time stated in Proposition \ref{lem3.1}.

Moreover, the strong solutions are continuously dependent on the initial data, in other words, for any two strong
solutions $(v_1,T_1)$ and $(v_2,T_2)$ to system (\ref{1.1-1})--(\ref{1.9-1}) on $\Omega\times(0,t_0)$, with initial data $(v_{10}, T_{10})$ and $(v_{20}, T_{20})$, respectively, it holds that
\begin{align*}
  &\sup_{0\leq t\leq t_0}(\|v\|_2^2+\|T\|_2^2)\leq Ce^{C\int_0^{t_0}(1+\|v_2\|_{H^2}^4+\|T_2\|_{H^2}^4)dt}(\|v_0\|_2^2+\|T_0\|_2^2),
\end{align*}
with $(v_0,T_0)=(v_{10}-v_{20}, T_{10},T_{20})$.
\end{proposition}

\begin{proof}
By Proposition \ref{lem3.1}, for any given $\varepsilon>0$, system (\ref{2.0-1})--(\ref{2.0-7}) has a solution $(v_\varepsilon, T_\varepsilon)$ in $\Omega\times(0,t_0^*)$ such that
$$
\sup_{0\leq t\leq t_0^*}(\|v_\varepsilon(t)\|_{H^2}^2+\|T_\varepsilon(t)\|_{H^2}^2)+\int_0^{t_0^*}(\varepsilon\|\nabla_H
T_\varepsilon\|_{H^2}^2+\|\partial_zT_\varepsilon\|_{H^2}^2+\|\nabla v_\varepsilon\|_{H^2}^2)dt\leq K
$$
and
$$
\int_0^{t_0^*}(\|\partial_tv_\varepsilon\|_{H^1}^2+\|\partial_t T_\varepsilon\|_{H^1}^2)dt\leq K,
$$
where $K$ is a constant which is independent of $\varepsilon$. On account of these estimates, by Lemma \ref{AL}, there is a subsequence of $(v_{\varepsilon_j},T_{\varepsilon_j})$ and $(v,T)$, such that
\begin{eqnarray*}
  &(v_{\varepsilon_j}, T_{\varepsilon_j})\rightarrow (v,T),\quad\mbox{ in }C([0,t_0];H^1(\Omega)),\\
  &v_{\varepsilon_j}\rightarrow v,\quad\mbox{ in }L^2(0,t_0;H^2(\Omega)),\qquad \partial_zT_{\varepsilon_j}\rightarrow\partial_zv,\quad\mbox{ in }L^2(0,t_0; H^1(\Omega)),\\
  &(v_{\varepsilon_j}, T_{\varepsilon_j})\rightharpoonup^*(v,T),\quad\mbox{ in }L^\infty(0,t_0; H^2(\Omega)),\\
  &v_{\varepsilon_j}\rightharpoonup v,\quad\mbox{ in }L^2(0,t_0;H^3(\Omega)),\qquad \partial_zT_{\varepsilon_j}\rightharpoonup\partial_zT,\quad\mbox{ in }L^2(0,t_0;H^2(\Omega)),\\
  &(\partial_tv_{\varepsilon_j},\partial_tT_{\varepsilon_j})\rightharpoonup(\partial_tv,\partial_tT),\quad
  \mbox{ in }L^2(0,t_0;H^1(\Omega)),
\end{eqnarray*}
where $\rightharpoonup$ and $\rightharpoonup^*$ are the weak and weak-* convergence, respectively. Due to these convergence, one can take the limit ${\varepsilon_j}\rightarrow0$ to see that $(v,T)$ is a strong solution to the system (\ref{1.10})--(\ref{1.16}), or equivalently system (\ref{1.1-1})--(\ref{1.9-1}). The regularities of $(v,T)$ follow from the uniform, in $\varepsilon$, estimates on $(v_{\varepsilon_j}, T_{\varepsilon_j})$ stated above and the weakly lower semi-continuity of the norms.

We now prove the continuous dependence on the initial data and the uniqueness of strong solutions. Let $(v_1,T_1)$ and $(v_2,T_2)$ be two strong solutions to system (\ref{1.1-1})--(\ref{1.9-1}) in $\Omega\times(0,t_0)$, with initial data $(v_{10},T_{10})$ and
$(v_{20},T_{20})$, respectively. Set
$v=v_1-v_2$ and $T=T_1-T_2$. One can easily check that $(v,T)$ satisfies the following system
\begin{eqnarray}
&\partial_tv+L_1v+(v_1\cdot\nabla_H)v+(v\cdot\nabla_H)v_2-\left(\int_{-h}^z\nabla_H\cdot v_1d\xi\right)\partial_zv-\left(\int_{-h}^z\nabla_H\cdot vd\xi\right)\partial_zv_2\nonumber\\
&+f_0k\times v+\nabla_Hp_s(x,y,t)-\nabla_H\left(\int_{-h}^zT(x,y,\xi,t)d\xi\right)=0,\label{3.5}\\
&\nabla_H\cdot\bar v=0,\nonumber\\
&\partial_tT+L_2T+v_1\cdot\nabla_HT+v\cdot\nabla_H T_2-\left(\int_{-h}^z\nabla_H\cdot v_1d\xi\right)\partial_zT\nonumber\\
&-\left(\int_{-h}^z\nabla_H\cdot vd\xi\right)\left(\partial_zT_2+\frac{1}{h}\right)=0\label{3.6}
\end{eqnarray}
and the boundary and initial conditions
\begin{eqnarray*}
&v \mbox{ and }T \mbox{ are periodic in }x, y, z, \\
&v \mbox{ and }T \mbox{ are even and odd in }z,\mbox{ respectively}, \\
&(v,T)|_{t=0}=(v_0,T_0).
\end{eqnarray*}

Multiplying (\ref{3.5}) by $v$ and integrating by parts (recalling the regularities of $v$ that we have just proved) yield
\begin{align}
&\frac{1}{2}\frac{d}{dt}\int_\Omega|v|^2dxdydz+\int_\Omega\left(\frac{1}{R_1}|\nabla_Hv|^2
+\frac{1}{R_2}|\partial_zv|^2\right)dxdydz\nonumber\\
=&\int_\Omega\left\{\left[\left(\int_{-h}^z\nabla_H\cdot vd\xi\right)\partial_zv_2-(v\cdot\nabla_H)v_2\right]
v-\left(\int_{-h}^zTd\xi\right)\nabla_H\cdot v\right\}dxdydz. \label{3.7}
\end{align}
By Lemma \ref{lem3.2} and using Cauchy's inequality, we have the following estimates
\begin{align*}
&\left|\int_\Omega\left(\int_{-h}^z\nabla_H\cdot vd\xi\right)\partial_zv_2\cdot vdxdydz\right|\leq\int_M\left(\int_{-h}^h|\nabla_Hv|dz\right)\left(\int_{-h}^h
|\partial_zv_2||v|dz\right)dxdy\\
\leq&C\|\nabla_Hv\|_2\|\partial_zv_2\|_2^{1/2}(\|\partial_zv_2\|_2^{1/2}+\|\nabla_H\partial_zv_2
\|_2^{1/2})\|v\|_2^{1/2}(\|v\|_2^{1/2}+\|\nabla_Hv\|_2^{1/2})\\
\leq&C\|v_2\|_{H^2}(\|v\|_2\|\nabla_Hv\|_2+\|v\|_2^{1/2}\|\nabla_Hv\|_2^{3/2})\leq \sigma\|\nabla_Hv\|_2^2+C_\sigma(1+\|v_2\|_{H^2}^4)\|v\|_2^2,
\end{align*}
with a sufficiently small positive constant $\sigma$.
Noticing that $|v_2(z)|\leq\frac{1}{2h}\int_{-h}^h|v_2(z)|dz+\int_{-h}^h|\partial_zv_2|dz$, applying Lemma \ref{lem3.2} again and using the Cauchy-Schwarz inequality, we then obtain
\begin{align*}
&\left|\int_\Omega (v\cdot\nabla_H)v_2\cdot vdxdydz\right|\\
=&-\int_\Omega[\nabla_H\cdot vv_2\cdot v+(v\cdot\nabla_H)v\cdot v_2]dxdydz\leq\int_\Omega|\nabla_Hv||v||v_2|dxdydz\\
\leq&C\int_M\left(\int_{-h}^h(|v_2|+|\partial_zv_2|)dz\right)\left(\int_{-h}^h|\nabla_Hv||v|dz
\right)dxdy\\
\leq&C\|\partial_zv_2\|_2^{1/2}(\|\partial_zv_2\|_2^{1/2}+\|\nabla_H\partial_zv_2\|_2^{1/2})
\|\nabla_Hv\|_2\|v\|_2^{1/2}(\|v\|_2^{1/2}+\|\nabla_Hv\|_2^{1/2})\\
&+C\|v_2\|_2^{1/2}(\|v_2\|_2^{1/2}+\|\nabla_H v_2\|_2^{1/2})
\|\nabla_Hv\|_2\|v\|_2^{1/2}(\|v\|_2^{1/2}+\|\nabla_Hv\|_2^{1/2})\\
\leq&\|v_2\|_{H^2}(\|v\|_2\|\nabla_Hv\|_2+\|v\|_2^{1/2}\|\nabla_Hv\|_2^{3/2})\leq \sigma\|\nabla_Hv\|_2^2+C_\sigma(1+\|v_2\|_{H^2}^4)\|v\|_2^2,
\end{align*}
with a sufficiently small positive constant $\sigma$.
Substituting these inequalities into (\ref{3.7}) implies
\begin{align}
&\frac{d}{dt}\int_\Omega|v|^2dxdydz+\int_\Omega\left(\frac{1}{R_1}|\nabla_Hv|^2
+\frac{1}{R_2}|\partial_zv|^2\right)dxdydz\nonumber\\
\leq&C(1+\|v_2\|_{H^2}^4)(\|v\|_2^2+\|T\|_2^2).\label{3.8}
\end{align}

Multiplying (\ref{3.6}) by $T$ and integrating by parts (recalling the regularities of $T$ that have just been proved) yield
\begin{align}
&\frac{1}{2}\frac{d}{dt}\int_\Omega|T|^2dxdydz+\frac{1}{R_3}\int_\Omega|\partial_zT|^2dxdydz
\nonumber\\
=&-\int_\Omega\left[v\cdot\nabla_HT_2-\left(\int_{-h}^z\nabla_H\cdot vd\xi\right)\left(\partial_zT_2+\frac{1}{h}\right)\right]Tdxdydz. \label{3.9}
\end{align}
Note that $|T(z)|\leq\frac{1}{2h}\int_{-h}^h|T(z)|dz+\int_{-h}^h|\partial_zT|dz$, we can apply Lemma \ref{lem3.2} and using the Cauchy-Schwarz inequality to deduce
\begin{align*}
&\left|\int_\Omega v\cdot\nabla_HT_2Tdxdydz\right|\\
\leq&\int_M\left(\int_{-h}^h(|T|+|\partial_zT|)dz\right)\left(\int_{-h}^h
|v||\nabla_HT_2|dz\right)dxdy\\
\leq&C(\|T\|_2+\|\partial_zT\|_2)\|v\|_2^{1/2}(\|v\|_2^{1/2}+\|\nabla_Hv\|_2^{1/2})\\
&\times\|\nabla_HT_2\|_2^{1/2}
(\|\nabla_HT_2\|_2^{1/2}+\|\nabla_H^2T_2\|_2^{1/2})\\
\leq&C(\|T\|_2+\|\partial_zT\|_2)(\|v\|_2+\|v\|_2^{1/2}\|\nabla_Hv\|_2^{1/2})\|T_2\|_{H^2}\\
\leq&\sigma(\|\partial_zT\|_2^2+\|\nabla_Hv\|_2^2)+C_\sigma(1+\|T_2\|_{H^2}^4)(\|v\|_2^2+\|T\|_2^2),
\end{align*}
and integration by parts yields
\begin{align*}
&\left|\int_\Omega\left(\int_{-h}^z\nabla_H\cdot vd\xi\right)\partial_zT_2Tdxdydz\right|\\
=&\left|-\int_\Omega\left[\nabla_H\cdot vT_2T+\left(\int_{-h}^z\nabla_H\cdot vd\xi\right)T_2\partial_zT\right]dxdydz\right|\\
\leq&\sigma\|\partial_zT\|_2^2+C_\sigma(1+\|T_2\|_\infty^2)(\|T\|_2^2+\|\nabla_Hv\|_2^2)\\
\leq&\sigma\|\partial_zT\|_2^2+C_\sigma(1+\|T_2\|_{H^2}^2)(\|T\|_2^2+\|\nabla_Hv\|_2^2),
\end{align*}
with a sufficiently small positive constant $\sigma$.
Substituting these estimates into (\ref{3.9}), one has
\begin{align}
&\frac{d}{dt}\int_\Omega|T|^2dxdydz+\frac{1}{R_3}\int_\Omega|\partial_zT|^2dxdydz
\nonumber\\
\leq&C(1+\|T_2\|_{H^2}^2)\|\nabla_Hv\|_2^2+C(1+\|T_2\|_{H^2}^4)(\|v\|_2^2+\|T\|_2^2). \label{3.10}
\end{align}

Since $(v_2, T_2)$ is a strong solution in $\Omega\times(0,t_0)$, it satisfies
$$
\sup_{0\leq t\leq t_0}(\|v_2\|_{H^2}^2+\|T_2\|_{H^2}^2)<\infty.
$$
Multiplying inequality (\ref{3.8}) by a sufficiently large positive number $\alpha$, and summing the resulting inequality with (\ref{3.10}), one reaches
\begin{align*}
\frac{d}{dt}\int_\Omega(\alpha|v|^2+|T|^2)dxdydz\leq C(1+\|v_2\|_{H^2}^4+\|T_2\|_{H^2}^4)(\|v\|_2^2+\|T\|_2^2).
\end{align*}
By the Gronwall inequality, it follows from this inequality that
\begin{align*}
  &\sup_{0\leq t\leq t_0}(\|v\|_2^2+\|T\|_2^2)\leq Ce^{C\int_0^{t_0}(1+\|v_2\|_{H^2}^4+\|T_2\|_{H^2}^4)dt}(\|v_0\|_2^2+\|T_0\|_2^2).
\end{align*}
This proves the continuous dependence of the initial data.
In particular, if $(v_{10},T_{10})=(v_{20},T_{20})$, then $(v,T)\equiv(0,0)$, i.e.
$(v_1,T_1)=(v_2,T_2)$, proving the uniqueness. This completes the proof.
\end{proof}

\section{Global Existence of Strong Solutions}\label{sec4}

In this section, we show that the local strong solution obtained in section \ref{sec3} can be
extended to be a global one. That is, we give below the proof of Theorem \ref{thm1}.

\begin{proof}[\textbf{Proof of Theorem \ref{thm1}}]
By Proposition \ref{prop2}, there is a unique strong solution $(v,T)$ to system (\ref{1.1-1})--(\ref{1.9-1}) on $\Omega\times(0,t_0^*)$ such that
\begin{equation}
\sup_{0\leq t\leq t_0^*}(\|T\|_{H^2}^2+\|v\|_{H^2}^2)+\int_0^{t_0^*}(\|\nabla v\|_{H^2}^2+\|\partial_zT\|_{H^2}^2)dt\leq C.\label{4.1}
\end{equation}

Set $u=\partial_zv$ and define functions $\eta$ and $\theta$
\begin{eqnarray}
\eta&=&\nabla_H^\bot\cdot u=\partial_xu^2-\partial_yu^1,\label{eta1}\\
\theta&=&\nabla_H\cdot u+R_1T=\partial_xu^1+\partial_yu^2+R_1T.\label{theta1}
\end{eqnarray}
By Proposition \ref{lem5.3} (see the Appendix section below), one has
\begin{align}\label{4.4}
&\sup_{0\leq s\leq t}[s^2(\|\eta(s)\|_{H^2}^2+\|\theta(s)\|_{H^2}^2)]+\int_0^ts^2(\|\eta(s)\|_{H^3}^2
\nonumber\\
&+\|\theta(s)\|_{H^3}^2+\|\partial_t\eta(s)\|_{H^1}^2+\|\partial_t\theta(s)\|_{H^1}^2)ds\leq K_3(t)
\end{align}
for all $t\in[0,t_0^*],$ where $K_3(t)$ is a bounded function on $[0,t_0^*]$.

We consider the strong solution $(v,T)$ on the maximal interval of existence $(0, T^*)$. We are going to prove that $T^*=\infty$. Suppose that $T^*<\infty$. Thanks to the regularity properties stated in Proposition \ref{lem5.3} and Proposition \ref{lem5.4}, below we define, for any $t\in(0,T^*),$
\begin{eqnarray*}
\mathcal X(t)&=&1+\|\nabla_H\Delta_H\bar v(t)\|_2^2+C_R\|\Delta_HT(t)\|_2^2+C_R\|\nabla_H\partial_zT(t)\|_2^2\\
&&+\|\Delta_H\eta(t)\|_2^2+\|\nabla_H\partial_z\eta(t)\|_2^2+\|\Delta_H\theta(t)\|_2^2
+\|\nabla_H\partial_z
\theta(t)\|_2^2\\
\mathcal Y(t)&=&\|\Delta_H^2\bar v(t)\|_2^2+\|\Delta_H\partial_zT(t)\|_2^2+\|\nabla_H\partial_z^2T(t)\|_2^2
+\|\nabla_H\Delta_H\eta(t)\|_2^2\\
&&+\|\Delta_H\partial_z\eta(t)\|_2^2+\|\nabla_H\Delta_H\theta(t)
\|_2^2+\|\Delta_H\partial_z\theta(t)\|_2^2,\\
\mathcal Z(t)&=&\log\mathcal X(t),
\end{eqnarray*}
where $C_R=\frac{2R_1^2(R_1+R_2)(R_2-R_3)^2}{R_2^2R_3}$. Therefore, it follows (see (122) in \cite{CAOTITI3})
\begin{align}
\frac{d\mathcal X}{dt}+C\mathcal Y\leq&C\|v\|_2(\|\nabla_H\bar v\|_{H^1(M)}+\|T\|_\infty+\|\eta\|_{H^1}+\|\theta\|_{H^1})\mathcal X\log\mathcal X\nonumber\\
&+\left[1+\|T\|_\infty^4+\|\partial_zT\|_2^2+\|\bar v\|_2^2(1+\|\bar v\|_{H^1}^2)+(1+\|\partial_zu\|_2^2)\|\partial_zu\|_{H^1}^2\right.\nonumber\\
&\left.+(1+\|\eta\|_2^2)\|\eta\|_{H^1}^2+(1+\|\theta\|_2^2)\|\theta\|_{H^1}^2\right]\mathcal X+(\|\eta\|_2^2\|\nabla_H\eta\|_2^2\nonumber\\
&+\|\theta\|_2^2\|\nabla_H\theta\|_2^2+\|T\|_\infty^4+\|\partial_z
u\|_2^2\|\nabla_H\partial_zu\|_2^2)\label{4.5}
\end{align}
and
\begin{align}
\frac{d\mathcal Z}{dt}\leq&C\|v\|_2(\|\nabla_H\bar v\|_{H^1(M)}+\|T\|_\infty+\|\eta\|_{H^1}+\|\theta\|_{H^1})\mathcal Z\nonumber\\
&+\left[1+\|T\|_\infty^4+\|\partial_zT\|_2^2+\|\bar v\|_2^2(1+\|\bar v\|_{H^1}^2)+(1+\|\partial_zu\|_2^2)\|\partial_zu\|_{H^1}^2\right.\nonumber\\
&\left.+(1+\|\eta\|_2^2)\|\eta\|_{H^1}^2+(1+\|\theta\|_2^2)\|\theta\|_{H^1}^2\right],\label{4.6}
\end{align}
for any $t\in(0,T^*)$.

Recalling (\ref{4.1}) and the definitions of $\eta,\theta$ in (\ref{eta1}) and (\ref{theta1}), it follows from (\ref{4.6}) that
\begin{align*}
\frac{d\mathcal Z}{dt}\leq&C(1+\|v\|_{H^3}^2)\mathcal Z+C(1+\|v\|_{H^3}^2),\quad\forall t\in(0,t_0^*),
\end{align*}
and thus
\begin{align}
&\frac{d}{dt}(t^2\mathcal Z(t))=t^2\frac{d\mathcal Z}{dt}+2t\mathcal Z(t)\nonumber\\
\leq&C(1+\|v\|_{H^3}^2)t^2\mathcal Z(t)+C(1+\|v\|_{H^3}^2)t^2+2t\mathcal Z(t)\label{4.7}
\end{align}
for $t\in(0,t_0^*)$.
Thanks to (\ref{4.1}) and (\ref{4.4}), it holds that $\int_0^{t_0^*}t\mathcal X(t)dt\leq C,$
and thus, noticing that $\mathcal X\geq1$, we have
$$
\int_0^{t_0^*}t\mathcal Z(t)dt=\int_0^{t_0^*}t\log(\mathcal X(t))dt\leq C\int_0^{t_0^*}t\mathcal X(t)dt\leq C.
$$
Combing the above with (\ref{4.7}) gives
\begin{equation*}
t^2\mathcal Z(t)\leq Ce^{C\int_0^t(1+\|v\|_{H^3}^2)ds}\int_0^t[(1+\|v\|_{H^3}^2)s^2+2s\mathcal Z(t)]ds\leq C,
\end{equation*}
for any $t\in(0,t_0^*)$. Recalling the definitions of $\mathcal X$ and $\mathcal Z$, the above inequality implies
\begin{equation}\label{4.8}
\sup_{\frac{t_0^*}{2}\leq t\leq{t_0^*}}\mathcal X(t)\leq C.
\end{equation}

Thanks to the estimates (59), (69), (91), (103) and (113) in \cite{CAOTITI3}, one has
\begin{align}
&\sup_{0\leq t\leq T^*}(\|v\|_2^2+\|T\|_\infty^2+\|\nabla_H\bar v\|_2^2+\|\partial_zu\|_2^2+\|\eta\|_2^2+\|\theta\|_2^2)\nonumber\\
&+\int_0^{T^*}(\|\nabla v\|_2^2+\|\partial_zT\|_2^2+\|\Delta_H\bar v\|_2^2+\|\nabla\partial_zu\|_2^2+\|\nabla\eta\|_2^2+\|\nabla\theta\|_2^2)dt\leq C, \label{4.9}
\end{align}
and consequently, we have
\begin{equation*}
\int_0^{t}\|v\|_2(\|\nabla_H\bar v\|_{H^1(M)}+\|T\|_\infty+\|\eta\|_{H^1}+\|\theta\|_{H^1})ds\leq C,
\end{equation*}
and
\begin{align*}
\int_0^t&\left[\|T\|_\infty^4+\|\partial_zT\|_2^2+\|\bar v\|_2^2(1+\|\bar v\|_{H^1}^2)+(1+\|\partial_zu\|_2^2)\|\partial_zu\|_{H^1}^2\right.\nonumber\\
&\left.+(1+\|\eta\|_2^2)\|\eta\|_{H^1}^2+(1+\|\theta\|_2^2)\|\theta\|_{H^1}^2\right]ds\leq C,
\end{align*}
for any $t\in(0,T_*)$. By the aid of these two inequalities, using (\ref{4.8}), it follows from (\ref{4.6}) that
$$
\sup_{t_0^*/2\leq t<T^* }\mathcal Z(t)\leq C.
$$
Thus by (\ref{4.5}) we have
\begin{equation}
\sup_{t_0^*/2\leq t<T^* }\mathcal X(t)+\int_{t_0^*/2}^{T^*}\mathcal Y(t)dt\leq C. \label{4.10}
\end{equation}

It is clear that
$$
v(x,y,z,t)=\bar v(x,y,t)+\frac{1}{2h}\int_{-h}^h\int_{z'}^z\partial_zv(x,y,\xi,t)d\xi dz',
$$
and thus, it follows from elliptic estimates and Poincar\'e's inequality that
\begin{align*}
&\|\nabla_H\nabla v\|_2^2=\|\nabla^2_Hv\|_2^2+\|\nabla_H\partial_zv\|_2^2=\|\nabla^2_Hv\|_2^2+\|\nabla_Hu\|_2^2\\
\leq&C\|\nabla_H^2\bar v\|_2^2+\|\nabla_H^2u\|_2^2+\|\nabla_Hu\|_2^2)\\
\leq&C(\|\Delta_H\bar v\|_2^2+\|\nabla_H(\nabla_H^\bot\cdot u)\|_2^2+\|\nabla_H(\nabla_H\cdot u)\|_2^2+\|\nabla_H^\bot\cdot u\|_2^2+\|\nabla_H\cdot u\|_2^2)\\
\leq&C(\|\nabla_H\Delta_H\bar v\|_2^2+\|\nabla_H\eta\|_2^2+\|\nabla_H\theta\|_2^2+\|\nabla_HT\|_2^2+\|\eta\|_2^2+\|\theta\|_2^2
+\|T\|_2^2)\\
\leq&C(\|\nabla_H\Delta_H\bar v\|_2^2+\|\nabla\nabla_H\eta\|_2^2+\|\nabla\nabla_H\theta\|_2^2+\|\nabla\nabla_HT\|_2^2+\|\eta\|_2^2+\|\theta\|_2^2
+\|T\|_2^2)\\
\leq&C(\|\nabla_H\Delta_H\bar v\|_2^2+\|\Delta_H\eta\|_2^2+\|\Delta_H\theta\|_2^2+\|\Delta_HT\|_2^2+\|\nabla_H
\partial_z\eta\|_2^2+\|\nabla_H\partial_z\theta\|_2^2\\
&+\|\nabla_H\partial_zT\|_2^2
+\|\eta\|_2^2+\|\theta\|_2^2+\|T\|_2^2)\\
\leq&C(\mathcal X(t)+\|\eta\|_2^2+\|\theta\|_2^2+\|T\|_2^2)
\end{align*}
and
\begin{align*}
&\|\nabla^2_H\nabla v\|_2^2=\|\nabla_H^2\partial_zv\|_2^2+\|\nabla^3_Hv\|_2^2=\|\nabla_H^2u\|_2^2+\|\nabla_H^3v\|_2^2\\
\leq&C(\|\nabla_H\nabla_H\cdot u\|_2^2+\|\nabla_H\nabla_H^\bot\cdot u\|_2^2+\|\nabla_H^3\bar v\|_2^2+\|\nabla_H^3u\|_2^2)\\
\leq&C(\|\nabla_H\eta\|_2^2+\|\nabla_H\theta\|_2^2+\|\nabla_HT\|_2^2+\|\nabla_H\Delta\bar v\|_2^2+\|\nabla_H^2\nabla_H\cdot u\|_2^2+\|\nabla_H^2\nabla_H^\bot\cdot u\|_2^2)\\
\leq&C(\|\nabla_H\eta\|_2^2+\|\nabla_H\theta\|_2^2+\|\nabla_H\nabla T\|_2^2+\|\nabla_H\Delta\bar v\|_2^2+\|\nabla_H^2\eta\|_2^2+\|\nabla_H^2\theta\|_2^2+\|\nabla_H^2T\|_2^2)\\
\leq&C(\|\nabla_H\eta\|_2^2+\|\nabla_H\theta\|_2^2+\|\nabla_H\partial_zT\|_2^2+\|\nabla_H\Delta\bar v\|_2^2+\|\Delta_H\eta\|_2^2+\|\Delta_H\theta\|_2^2+\|\Delta_HT\|_2^2)\\
\leq&C(\mathcal X(t)+\|\nabla_H\eta\|_2^2+\|\nabla_H\theta\|_2^2).
\end{align*}
Combining these two estimates it follows from (\ref{4.9}) and (\ref{4.10}) that
\begin{align*}
&\sup_{t_0^*/2\leq t<T^*}\|v(t)\|_{H^2}^2+\int_{t_0^*/2}^{T^*}\|\nabla^3v(t)\|_2^2dt\nonumber\\
\leq&C\sup_{t_0^*/2\leq t<T^*}(\|v\|_2^2+\|\partial_z^2v\|_2^2+\|\nabla_H\nabla v\|_2^2)+C\int_{t_0^*/2}^{T^*}(+\|\nabla\partial_z^2v\|_2^2\|\nabla_H^2\nabla v\|_2^2)dt\nonumber\\
=&C\sup_{t_0^*/2\leq t<T^*}(\|v\|_2^2+\|\partial_zu\|_2^2+\mathcal X(t)+\|\eta\|_2^2+\|\theta\|_2^2+\|T\|_2^2)\nonumber\\
&+C\int_{t_0^*/2}^{T^*}(\|\nabla\partial_zu\|_2^2+\mathcal X(t)+\|\nabla_H\eta\|_2^2+\|\nabla_H\theta\|_2^2)dt\leq C,
\end{align*}
and thus
\begin{equation}\label{4.13}
\sup_{t_0^*/2\leq t<T^*}\|v\|_{H^2}^2+\int_{t_0^*/2}^{T^*}\|v\|_{H^3}^2dt\leq C.
\end{equation}

Thanks to (\ref{4.9}), (\ref{4.10}) and using the Poincar\'e inequality, it follows that
\begin{equation}\label{4.14}
\sup_{{t_0^*/2}\leq t< T^*}(\|T\|_2^2+\|\nabla_HT\|_{H^1}^2)+\int_{t_0^*/2}^{T^*}(\|\partial_zT\|_2^2+\|\partial_z\nabla_HT\|_{H^1}^2)dt\leq C.
\end{equation}
Applying the operator $\partial_z$ to (\ref{1.12}) and multiplying the resulting equation by $-\partial_z^3T$, then it follows after integrating by parts and using Lemma \ref{lem2.0} that
\begin{align*}
&\frac{1}{2}\frac{d}{dt}\int_\Omega|\partial_z^2T|^2dxdydz+\frac{1}{R_3}\int_\Omega|\partial_z^3T|^2
dxdydz\\
=&-\int_\Omega\left((\partial_z^2v\cdot\nabla_H)T+2(\partial_zv\cdot\nabla_H)\partial_zT-(\partial_z
\nabla_H
\cdot v)\left(\partial_zT+\frac{1}{h}\right)\right.\\
&-2(\nabla_H\cdot v)\partial_z^2T\bigg)\partial_z^2Tdxdydz\\
\leq&C(\|\partial_z^2v\|_3\|\nabla_HT\|_6+\|\partial_zv\|_\infty\|\nabla_H\partial_zT\|_2+\|\partial_z
\nabla_Hv\|_3\|\partial_zT\|_6\\
&+\|\partial_z\nabla_Hv\|_2+\|\nabla_Hv\|_\infty\|\partial_z^2T\|_2)
\|\partial_z^2T\|_2\\
\leq&C(\|v\|_{H^3}\|\nabla_H\nabla T\|_2+\|v\|_{H^3}\|\nabla\partial_zT\|_2+\|v\|_{H^2})\|\partial_z^2T\|_2\\
\leq&C(\|v\|_{H^3}\|\nabla_HT\|_{H^1}+\|v\|_{H^3}\|\partial_z^2T\|_2+\|v\|_{H^2})\|\partial_z^2T\|_2
\\
\leq&C\|v\|_{H^3}(\|\nabla_HT\|_{H^1}+1)+C\|v\|_{H^3}\|\partial_z^2T\|_2^2.
\end{align*}
By the aid of (\ref{4.1}), (\ref{4.13}) and (\ref{4.14}), it follows from the above inequality that
\begin{align}
&\sup_{0\leq t<T^*}\|\partial_z^2T\|_2^2+\int_{0}^{T^*}\|\partial_z^3T\|_2^2dt\nonumber\\
\leq&Ce^{\int_0^{T^*}\|v\|_{H^3}dt}\left(\|\partial_z^2T_0\|_2^2+\int_0^t\|v\|_{H^3}
(1+\|\nabla_HT\|_{H^2}^2)ds\right)\leq C. \label{4.15}
\end{align}

Combining (\ref{4.1}) with (\ref{4.13})--(\ref{4.15}), we obtain
\begin{equation*}
\sup_{0\leq t< T^*}(\|v\|_{H^2}^2+\|T\|_{H^2}^2)+\int_0^{T^*}(\|v\|_{H^3}^2+\|T\|_{H^3}^2)dt\leq C.
\end{equation*}
As a result, we apply Proposition \ref{prop2} to extend the strong solution $(v,T)$ beyond $T^*$, contradicting to the fact that $T^*$ is a finite maximal time of existence. This contradiction implies that $T^*=\infty$, and this completes the proof.
\end{proof}

\section{Appendix: regularities}\label{sec5}

In this appendix, we justify some necessary regularities used in the previous section, section \ref{sec4}. Let $(v,T)$ be a strong solution to system (\ref{1.10})--(\ref{1.16}) in $\Omega\times(0,t_0)$ with $0<t_0<\infty$. For any $t\in[0,t_0)$, we set
$$
K_0(t)=\sup_{0\leq s\leq t}(\|v\|_{H^2}^2+\|T\|_{H^2}^2)+\int_0^t(\|\nabla v\|_{H^2}^2+\|\partial_zT\|_{H^2}^2)ds.
$$
Recall the definition of the functions $u, \eta, \theta$ and $\zeta$
\begin{eqnarray}
&&u=\partial_zv,\qquad \zeta=\partial_zu, \nonumber\\
&&\eta=\nabla_H^\bot\cdot u=\partial_xu^2-\partial_yu^1,\label{eta}\\
&&\theta=\nabla_H\cdot u+R_1T=\partial_xu^1+\partial_yu^2+R_1T.\label{theta}
\end{eqnarray}
We are going to study the regularities of $u,\zeta,\eta,\theta$ and $\bar v$.

For convenience, we also use the notation $\textbf{x}=(x^1,x^2,x^3)$ to denote the spacial variables, that is using $x^1,x^2$ and $x^3$ to replace $x,y$ and $z$, respectively. We will use both $\textbf{x}$ and $(x,y,z)$ to denote the spacial variables. Set $e_1=(1,0,0)$, $e_2=(0,1,0)$ and $e_3=(0,0,1)$. For any spatial periodic function $f$ and $l\not=0$, we define the difference quotient operators $\delta_l^i, i=1,2,3$ as follows
$$
\delta_l^if(\textbf{x})=\frac{1}{l}(f(\textbf{x}+le_i)-f(\textbf{x})).
$$
Since we will use $\delta_l^3$ more frequently than $\delta_l^i, i=1,2$, we will often use $\delta_l$ instead of $\delta_l^3$.

Straightforward calculations show, for any periodic functions $f$ and $g$, that
\begin{eqnarray*}
&\delta_{l_1}\delta_{l_2}=\delta_{l_2}\delta_{l_1},\quad\nabla\delta_l=\delta_l\nabla,\\
&\delta_l^*=-\delta_{-l},\quad\mbox{ i.e. }\quad\int_\Omega\delta_lfgd\textbf{x}=-\int_\Omega f\delta_{-l}gd\textbf{x},\\
&\delta_l(fg)=f\delta_lg+g(\cdot+le_3)\delta_lf,\\
&\|\delta_lf\|_p\leq C\|\partial_zf\|_p,\quad\forall f\in W^{1,p}(\Omega), 1\leq p\leq\infty,
\end{eqnarray*}
where $\delta_l^*$ stands for the adjoint operator of $\delta_l$, and $C$ is an absolute constant.
The operators $\delta_l^i, i=1,2$ have the same properties as those of $\delta_l$.

\begin{lemma}
  \label{lem0}
Suppose that the spatial periodic function $f$ satisfies
$$
\sup_{0\leq s\leq t}(s^k\|\delta_l^if\|_2^2)+\int_0^{t}s^k\|\nabla\delta_l^if\|_2^2dt\leq K(t),
$$
for any $t\in(0,t_0)$ and for any $0<|l|<1$, where $K(t)$ is a bounded increasing function on $(0,t_0)$. Then
$$
\sup_{0\leq s\leq t}(s^k\|\partial_if\|_2^2)+\int_0^{t}s^k\|\nabla\partial_if\|_2^2dt\leq K(t),
$$
for any $t\in(0,t_0)$.
\end{lemma}

\begin{proof} Set $g=s^{k/2}f$, then
$$
\sup_{0\leq s\leq t}\|\delta_l^ig\|_2^2+\int_0^{t}\|\nabla\delta_l^ig\|_2^2dt\leq K(t),
$$
for any $t\in(0,t_0)$. Given $t\in(0,t_0)$, using the sequentially weak compactness of closed balls in $L^2(\Omega)$ and $L^2(\Omega\times(0,t))$, any sequence $l_n$ with $l_n\rightarrow0$ as $n\rightarrow\infty$ has a subsequence, still denoted by $l_n$, such that
$$
\delta_{l_n}^ig(\cdot,t)\rightharpoonup\Phi,\mbox{ in }L^2(\Omega),
$$
and
$$
\delta_{l_n}^i\nabla g\rightharpoonup\Psi,\mbox{ in }L^2(\Omega\times(0,t)).
$$
Using the properties of $\delta_l^i$ stated above, for any two spatial periodic functions $\phi\in C^\infty(\mathbb R^3)$ and $\psi\in C^\infty(\mathbb R^3\times[0,t])$, one has
$$
\int_\Omega\delta_{l_n}^ig(\cdot,t)\phi d\textbf{x}=-\int_\Omega g(\cdot,t)\delta_{-l_n}^i\phi d\textbf{x}\rightarrow
-\int_\Omega g(\cdot,t)\partial_i\phi d\textbf{x}
$$
and
$$
\int_0^t\int_\Omega\delta_{l_n}^i\nabla g\psi d\textbf{x}ds=-\int_0^t\int_\Omega \nabla g\delta_{-l_n}^i\psi d\textbf{x}ds\rightarrow-\int_0^t\int_\Omega \nabla g\partial_i\psi d\textbf{x}ds.
$$
On the other hand, the weak convergence of $\delta_{l_n}^ig(\cdot,t)$ and $\delta_{l_n}^i\nabla g$ in $L^2(\Omega)$ and $L^2(\Omega\times(0,t))$, respectively, implies
$$
\int_\Omega\delta_{l_n}^ig(\cdot,t)\phi d\textbf{x}\rightarrow\int_\Omega\Phi\phi d\textbf{x},
$$
and
$$
\int_0^t\int_\Omega\delta_{l_n}\nabla g\psi d\textbf{x}ds\rightarrow\int_0^t\int_\Omega\Psi\psi d\textbf{x}ds.
$$
Therefore, we have
$$
\int_\Omega\Phi\phi dx=-\int_\Omega g(\cdot,t)\partial_i\phi d\textbf{x},\quad\int_0^t\int_\Omega\Psi\psi d\textbf{x}ds=-\int_0^t\int_\Omega \nabla g\partial_i\psi d\textbf{x}ds,
$$
which imply
$$
\Phi=\partial_ig(\cdot,t),\qquad\Psi=\partial_i\nabla g.
$$
Combing the above statements, we have proven that for any sequence $l_n$, with $l_n\rightarrow0$, it has a subsequence, still denoted by $l_n$, such that
$$
\delta_{l_n}^ig(\cdot,t)\rightharpoonup\partial_ig(\cdot,t),\mbox{ in }L^2(\Omega),
$$
and
$$
\delta_{l_n}^i\nabla g\rightharpoonup\partial_i\nabla g,\mbox{ in }L^2(\Omega\times(0,t)).
$$
Thanks to the above two weak convergence, using the weakly lower semi-continuity of the norms and by assumption, we have
\begin{align*}
\|\partial_ig(\cdot,t)\|_2^2+\int_0^t\|\partial_i\nabla g\|_2^2ds
  \leq \varliminf_{n\rightarrow\infty}\left(\|\delta_{l_n}^ig(\cdot,t)\|_2^2+\int_0^t\|\delta_{l_n}^i
  \nabla g\|_2^2ds\right)\leq K(t),
\end{align*}
which, recalling the definition of $g$, gives
\begin{equation*}
  t^k\|\partial_if(\cdot,t)\|_2^2+\int_0^ts^k\|\partial_i\nabla f\|_2^2ds\leq K(t),
\end{equation*}
for any $t\in(0,t_0)$. The conclusion follows from taking the supremum with respect to time $t$. This completes the proof.
\end{proof}

We first consider the regularities of $u$, that is the following proposition.

\begin{proposition}\label{lem5.1}
Let $(v,T)$ be a strong solution to system (\ref{1.10})--(\ref{1.16}) in $\Omega\times(0,t_0)$ and set $u=\partial_zv$. Then
\begin{equation*}
\sqrt tu\in L^\infty(0,t_0; H^2(\Omega))\cap L^2(0,t_0; H^3(\Omega)),
\quad\sqrt t\partial_tu\in L^2(0,t_0; H^1(\Omega))
\end{equation*}
and
\begin{equation*}
\sup_{0\leq s\leq t}(s \|  u\|_{H^2}^2)+\int_0^ts (\|  u\|_{H^3}^2+\|\partial_tu\|_{H^1}^2)ds
\leq K_1(t)
\end{equation*}
for any $t\in(0,t_0)$, where $K_1(t)$ is a bounded increasing function on $(0,t_0)$.
\end{proposition}

\begin{proof}
Strong solution $(v,T)$ has the following regularity properties
$v\in L^2(0,t_0; H^3(\Omega))$ and $\partial_tv\in L^2(0,t_0; H^1(\Omega))$. One can differentiate equation (\ref{1.10}) with respect to $z$ to derive
\begin{eqnarray}
&\partial_tu+L_1u+(v\cdot\nabla_H)u-\left(\int_{-h}^z\nabla_H\cdot v(x,y,\xi,t)d\xi\right)\partial_zu\nonumber\\
&+(u\cdot\nabla_H)v-(\nabla_H\cdot v)u+f_0k\times u-\nabla_HT=0. \label{nnw1}
\end{eqnarray}
Note that $\partial_tu\in L^2(0,t_0;L^2(\Omega))$ and $u\in L^2(0,t_0; H^2(\Omega))$. Multiplying the above equation by $-\delta_l^*\delta_l\Delta u$ and applying Lemma \ref{lem2.0}, it follows from the Sobolev embedding inequality and the Cauchy-Schwarz inequality that
\begin{align*}
&\frac{1}{2}\frac{d}{dt}\int_\Omega|\delta_l\nabla u|^2d\textbf{x}+\int_\Omega\left(\frac{1}{R_1}|\delta_l\nabla_H\nabla u|^2+\frac{1}{R_2}|\delta_l\partial_z\nabla u|^2\right)d\textbf{x}\\
=&\int_\Omega[(u\cdot\nabla_H)v-(\nabla_H\cdot v)u-\nabla_HT]\delta_l^*\delta_l\Delta ud\textbf{x}\\
&+\int_\Omega\left[(v\cdot\nabla_H)u-\left(\int_{-h}^z\nabla_H\cdot v(x,y,\xi,t)d\xi\right)\partial_zu\right]\delta_l^*\delta_l\Delta ud\textbf{x}\\
=&\int_\Omega\delta_l[(u\cdot\nabla_H)v-(\nabla_H\cdot v)u-\nabla_HT] \delta_l\Delta ud\textbf{x}\\
&+\int_\Omega\left[(\delta_lv\cdot\nabla_H)u(\textbf{x}+le_3,t)-\delta_l\left(\int_{-h}^z\nabla_H\cdot vd\xi\right)\partial_zu(\textbf{x}+le_3,t)\right]\delta_l\Delta ud\textbf{x}\\
&+\int_\Omega\left[(v\cdot\nabla_H)\delta_lu-\left(\int_{-h}^z\nabla_H\cdot v(x,y,\xi,t)d\xi\right)\partial_z\delta_lu\right]\delta_l\Delta ud\textbf{x}\\
=&\int_\Omega\delta_l[(u\cdot\nabla_H)v-(\nabla_H\cdot v)u-\nabla_HT]\delta_l\Delta ud\textbf{x}\\
&+\int_\Omega\left[(\delta_lv\cdot\nabla_H)u(\textbf{x}+le_3,t)-\delta_l\left(\int_{-h}^z\nabla_H\cdot vd\xi\right)\partial_zu(\textbf{x}+le_3,t)\right]\delta_l\Delta ud\textbf{x}\\
&-\int_\Omega\left[(\nabla v\cdot\nabla_H)\delta_lu-\left(\int_{-h}^z\nabla\nabla_H\cdot v(x,y,\xi,t)d\xi\right)\partial_z\delta_lu\right]\delta_l\nabla ud\textbf{x}\\
&-\int_\Omega\left[(v\cdot\nabla_H)\delta_l\nabla u-\left(\int_{-h}^z\nabla_H\cdot v(x,y,\xi,t)d\xi\right)\partial_z\delta_l\nabla u\right]\delta_l\nabla ud\textbf{x}\\
=&\int_\Omega\delta_l[(u\cdot\nabla_H)v-(\nabla_H\cdot v)u-\nabla_HT]\delta_l\Delta ud\textbf{x}\\
&+\int_\Omega\left[(\delta_lv\cdot\nabla_H)u(\textbf{x}+le_3,t)-\delta_l\left(\int_{-h}^z\nabla_H\cdot vd\xi\right)\partial_zu(\textbf{x}+le_3,t)\right]\delta_l\Delta ud\textbf{x}\\
&-\int_\Omega\left[(\nabla v\cdot\nabla_H)\delta_lu-\left(\int_{-h}^z\nabla\nabla_H\cdot v(x,y,\xi,t)d\xi\right)\partial_z\delta_lu\right]\delta_l\nabla ud\textbf{x}\\
\leq&C\|\delta_l[(u\cdot\nabla_H)v-(\nabla_H\cdot v)u]\|_2\|\delta_l\Delta u\|_2+C\|\delta_l\nabla_HT\|_2\|\delta_l\Delta u\|_2\\
&+C\bigg(\|\delta_lv\|_\infty\|\nabla_Hu\|_2
+\bigg\|\delta_l\bigg(\int_{-h}^z\nabla_H\cdot vd\xi\bigg)\bigg\|_\infty\|\partial_zu\|_2\bigg)\|\delta_l\Delta u\|_2\\
&+C(\|\nabla v\|_\infty\|\delta_l\nabla u\|_2^2+\|\nabla^2v\|_3\|\delta_l\nabla u\|_2\|\delta_l\nabla u\|_6)\\
\leq&C\|\partial_z[(u\cdot\nabla_H)v-(\nabla_H\cdot v)u]\|_2\|\delta_l\Delta u\|_2+C\|\partial_z\nabla_HT\|_2\|\delta_l\Delta u\|_2\\
&+\left(\|\nabla v\|_\infty\|\nabla_Hu\|_2+\left\|\partial_z\left(\int_{-h}^z\nabla_H\cdot vd\xi\right)\right\|_\infty\|\partial_zu\|_2\right)\|\delta_l\Delta u\|_2\\
&+C(\|\nabla v\|_\infty\|\delta_l\nabla u\|_2^2+\|\nabla^2v\|_3\|\delta_l\nabla u\|_2\|\delta_l\nabla^2u\|_2)\\
\leq&C\left[\int_\Omega(|\nabla u|^2|\nabla v|^2+|u|^2|\nabla^2v|^2)d\textbf{x}\right]^{1/2}\|\delta_l\Delta u\|_2+C\|\nabla^2T\|_2\|\delta_l\Delta u\|_2\\
&+C(\|v\|_{H^3}\|v\|_{H^2}\|\delta_l\Delta u\|_2+\|v\|_{H^3}\|\delta_l\nabla u\|_2^2+\|v\|_{H^3}\|\delta_l\nabla u\|_2\|\delta_l\nabla^2u\|_2)\\
\leq&C(\|\nabla v\|_\infty\|\nabla u\|_2+\|u\|_\infty\|\nabla^2v\|_2+\|\nabla^2T\|_2)\|\delta_l\Delta u\|_2\\
&+C(\|v\|_{H^3}\|v\|_{H^2}\|\delta_l\Delta u\|_2+\|v\|_{H^3}\|\delta_l\nabla u\|_2^2+\|v\|_{H^3}\|\delta_l\nabla u\|_2\|\delta_l\nabla^2u\|_2)\\
\leq&C(\|v\|_{H^2}\|v\|_{H^3}+\|T\|_{H^2})\|\delta_l\Delta u\|_2+C(\|v\|_{H^3}\|v\|_{H^2}\|\delta_l\Delta u\|_2\\
&+\|v\|_{H^3}\|\delta_l\nabla u\|_2^2+\|v\|_{H^3}\|\delta_l\nabla u\|_2\|\delta_l\nabla^2u\|_2)\\
\leq&\sigma\|\delta_l\nabla^2 u\|_2^2+C_\sigma(1+\|v\|_{H^3}^2)(\|\delta_l\nabla u\|_2^2+\|v\|_{H^2}^2+\|T\|_{H^2}^2),
\end{align*}
with a sufficiently small positive constant $\sigma$, and thus
\begin{align*}
\frac{d}{dt}\|\delta_l\nabla u\|_2^2+\frac{1}{R}\|\delta_l\nabla^2u\|_2^2
\leq&C(1+\|v\|_{H^3}^2)(\|\delta_l\nabla u\|_2^2+\|v\|_{H^2}^2+\|T\|_{H^2}^2)
\end{align*}
with $R=R_1+R_2$,
from which we obtain
\begin{align*}
&\frac{d}{dt}(t\|\delta_l\nabla u\|_2^2)+\frac{1}{R}t\|\delta_l\nabla^2u\|_2^2\\
\leq&C(1+\|v\|_{H^3}^2)t(\|\delta_l\nabla u\|_2^2+\|v\|_{H^2}^2+\|T\|_{H^2}^2)+\|\delta_l\nabla u\|_2^2\\
\leq&C(1+\|v\|_{H^3}^2)t(\|\delta_l\nabla u\|_2^2+\|v\|_{H^2}^2+\|T\|_{H^2}^2)+C\|\partial_z\nabla u\|_2^2\\
\leq&C(1+\|v\|_{H^3}^2)t\|\delta_l\nabla u\|_2^2+C(1+\|v\|_{H^3}^2)(\|v\|_{H^2}^2+\|T\|_{H^2}^2)(t+1).
\end{align*}
Integrating this inequality with respect to $t$ and applying Lemma \ref{lem0}, we obtain $\sqrt t\partial_zu\in L^\infty(0,t_0;H^1(\Omega))\cap L^2(0,t_0; H^2(\Omega))$ and
\begin{align*}
&\sup_{0\leq s\leq t}(s\|\nabla\partial_zu\|_2^2)+\int_0^ts\|\nabla^2\partial_zu\|_2^2ds\\
\leq&Ce^{C\int_0^t(1+\|v\|_{H^3}^2)ds}\int_0^t(1+\|v\|_{H^3}^2)(\|v\|_{H^2}^2+\|T\|_{H^2}^2)
(s+1)ds\\
\leq&Ce^{C(t+K_0(t))}(K_0(t)t+K_0^2(t))(t+1),
\end{align*}
and therefore
\begin{align}
&\sup_{0\leq s\leq t}(s\| \partial_zu\|_{H^1}^2)+\int_0^ts\|\partial_zu\|_{H^2}^2ds\leq K_1'(t)\label{nnw2}
\end{align}
for all $t\in(0,t_0)$.

Set
\begin{align*}
f_1(x,y,z,t)=&-(v\cdot\nabla_H)u+\left(\int_{-h}^z\nabla_H\cdot v(x,y,\xi,t)d\xi\right)\partial_zu\nonumber\\
&-(u\cdot\nabla_H)v+(\nabla_H\cdot v)u-f_0k\times u+\nabla_HT.
\end{align*}
Then it follows that
\begin{align*}
\|\nabla f_1\|_2^2\leq&\int_\Omega\bigg[|v|^2|\nabla^2u|^2+|\nabla v|^2|\nabla u|^2+\left(\int_{-h}^h|\nabla^2 v|d\xi\right)^2|\partial_zu|^2\\
&+\left(\int_{-h}^h|\nabla v|d\xi\right)^2|\nabla\partial_zu|^2+|u|^2|\nabla^2v|^2+|\nabla u|^2+|\nabla^2T|^2\bigg]d\textbf{x}\\
\leq&C(\|v\|_\infty^2\|\nabla^2u\|_2^2+\|\nabla v\|_\infty^2\|\nabla u\|_2^2+\|\nabla^2v\|_2^2\|\partial_zu\|_\infty^2\\
&+\|\nabla v\|_\infty^2\|\nabla\partial_zu\|_2+\|u\|_\infty^2\|\nabla^2v\|_2^2+\|\nabla u\|_2^2+\|\nabla^2T\|_2^2)\\
\leq&C(\|v\|_{H^2}^2\|v\|_{H^3}^2+\|v\|_{H^2}^2\|\partial_z u\|_{H^2}^2+\|v\|_{H^3}^2\|\partial_zu\|_{H^1}^2+\|v\|_{H^2}^2+\|T\|_{H^2}^2),
\end{align*}
and thus
\begin{align}
\int_0^ts\|\nabla f_1\|_2^2ds\leq&C\int_0^ts(\|v\|_{H^2}^2\|v\|_{H^3}^2+\|v\|_{H^3}^2\|\partial_zu\|_{H^1}^2
\nonumber\\
&+\|v\|_{H^2}^2\|\partial_z u\|_{H^2}^2+\|v\|_{H^2}^2+\|T\|_{H^2}^2)ds\nonumber\\
=&C\int_0^ts(\|v\|_{H^2}^2\|v\|_{H^3}^2+\|v\|_{H^2}^2+\|T\|_{H^2}^2)ds\nonumber\\
&+C\int_0^ts(\|v\|_{H^3}^2\|\partial_zu\|_{H^1}^2+\|v\|_{H^2}^2\|\partial_z u\|_{H^2}^2)ds\nonumber\\
\leq&Ct(K_0^2(t)+K_0(t)t)+CK_0(t)K_1'(t)\nonumber\\
=&CK_0(t)(K_0(t)t+K_1'(t)+t^2)=:K_1''(t).\label{nnw3}
\end{align}
Note that $u$ satisfies equation
\begin{equation}\label{nnw5}
\partial_tu+L_1u=f_1.
\end{equation}
Recalling that $u\in L^2(0,t_0; H^2(\Omega))$ and $\partial_tu\in L^2(0;t_0; L^2(\Omega))$. Multiplying the above equation by $-(\delta_l^i)^*\delta_l^i\Delta u, i=1,2$, by Lemma \ref{lem2.0}, we obtain
\begin{align*}
&\frac{d}{dt}\int_\Omega|\delta_l^i\nabla u|^2dxdydz+\int_\Omega\left(\frac{1}{R_1}|
\delta_l^i\nabla_H\nabla u|^2+\frac{1}{R_2}|\delta_l^i\partial_z\nabla u|^2\right)dxdydz\\
=&-\int_\Omega\delta_l^if_1\delta_l^i\Delta u\leq\eta\int_\Omega|\delta_l^i\Delta u|^2dxdydz
+C\int_\Omega|\delta_l^if_1|^2dxdydz,
\end{align*}
and thus
\begin{equation*}
\frac{d}{dt}\|\delta_l^i\nabla u\|_2^2+\frac{1}{R}\|\delta_l^i\nabla^2 u\|_2^2\leq C\|\delta_l^if_1\|_2^2\leq C\|\partial_if_1\|_2^2
\end{equation*}
with $R=R_1+R_2$, from which we obtain
\begin{align*}
\frac{d}{dt}(t\|\delta_l^i\nabla u\|_2^2)+\frac{1}{R}t\|\delta_l^i\nabla^2 u\|_2^2\leq
Ct\|\nabla f_1\|_2^2+\|\delta_l^i\nabla u\|_2^2.
\end{align*}
Integrating this inequality in $t$, thanks to (\ref{nnw3}) and applying Lemma \ref{lem0}, we then obtain
\begin{align*}
&\sup_{0\leq s\leq t}(s \|\nabla_H\nabla u\|_2^2)+\int_0^ts \|\nabla_H\nabla^2 u\|_2^2ds\\
\leq&C\int_0^ts \|\nabla f_1\|_2^2ds+C\int_0^ts\|\nabla_H\nabla u\|_2^2ds\leq C(K_0(t)+K_1''(t)).
\end{align*}
Combining this with (\ref{nnw2}) and using (\ref{nnw3}), (\ref{nnw5}), we obtain
\begin{equation}\label{nnw4}
\sup_{0\leq s\leq t}(s \|  u\|_{H^2}^2)+\int_0^ts (\|  u\|_{H^3}^2+\|\partial_tu\|_{H^1}^2)ds
\leq K_1(t)
\end{equation}
for all $t\in(0,t_0)$. This completes the proof.
\end{proof}

Next, we establish the regularity properties of $\zeta$ as stated in the following proposition.

\begin{proposition}\label{lem5.2}
Let $(v,T)$ be a strong solution to system (\ref{1.10})--(\ref{1.16}) in $\Omega\times(0,t_0)$ and set $\zeta=\partial_z^2v$. Then
\begin{equation*}
t\zeta\in L^\infty(0,t_0; H^2(\Omega))\cap L^2(0,t_0; H^3(\Omega)),
\quad t\partial_t\zeta\in L^2(0,t_0; H^1(\Omega))
\end{equation*}
and
\begin{equation*}
\sup_{0\leq s\leq t}(s^2\| \zeta\|_{H^2}^2)+\int_0^ts^2(\| \zeta\|_{H^3}^2+\|\partial_t\zeta\|_{H^1}^2)ds,
\leq K_2(t)
\end{equation*}
for any $t\in(0,t_0)$, where $K_2(t)$ is an increasing bounded function on $(0,t_0)$.
\end{proposition}

\begin{proof}
By Proposition \ref{lem5.1}, one can differentiate equation (\ref{nnw1}) with respect to $z$ to deduce
\begin{eqnarray}
&\partial_t\zeta+L_1\zeta+f_0k\times\zeta=\left(\int_{-h}^z\nabla_H\cdot vd\xi\right)\partial_z\zeta+\nabla_H\cdot v\partial_zu\nonumber\\
&+\nabla_H\partial_zT-\partial_z[(v\cdot\nabla_H)u+(u\cdot\nabla_H)v-(\nabla_H\cdot v)u].\label{nw5}
\end{eqnarray}
Multiplying this equation by $-\delta_l^*\delta_l\Delta\zeta$, then it follows from Lemma \ref{lem2.0}, the Sobolev embedding inequality and the Cauchy-Schwarz inequality that
\begin{align*}
&\frac{1}{2}\frac{d}{dt}\int_\Omega|\delta_l\nabla\zeta|^2d\textbf{x}+\int_\Omega\left(
\frac{1}{R_1}|\delta_l\nabla_H\nabla\zeta|^2+\frac{1}{R_2}|\delta_l\partial_z\nabla\zeta|^2
\right)d\textbf{x}\\
=&\int_\Omega\delta_l\big\{\partial_z[(v\cdot\nabla_H)u+(u\cdot\nabla_H)v-(\nabla_H\cdot v)u]-\nabla_H\cdot v\partial_zu-\nabla_H\partial_zT\big\}\delta_l\Delta\zeta d\textbf{x}\\
&-\int_\Omega\delta_l\left[\left(\int_{-h}^z\nabla_H\cdot vd\xi\right)\partial_z\zeta\right]\delta_l\Delta\zeta d\textbf{x}\\
\leq&C\|\delta_l\partial_z[(v\cdot\nabla_H)u+(u\cdot\nabla_H)v-(\nabla_H\cdot v)u]-\delta_l(\nabla_H\cdot v\partial_zu+\nabla_H\partial_zT)\|_2\|\delta_l\Delta\zeta\|_2\\
&+\int_\Omega\left[\delta_l\left(\int_{-h}^z\nabla_H\cdot vd\xi\right)\partial_z\zeta(\textbf{x}+le_3,t)+\left(\int_{-h}^z\nabla_H\cdot vd\xi\right)\delta_l\partial_z\zeta\right]\delta_l\Delta\zeta d\textbf{x}\\
\leq&C\|\partial_z^2[(v\cdot\nabla_H)u+(u\cdot\nabla_H)v-(\nabla_H\cdot v)u]-\partial_z(\nabla_H\cdot v\partial_zu+\nabla_H\partial_zT)\|_2\|\delta_l\Delta\zeta\|_2\\
&+C\left(\left\|\delta_l\left(\int_{-h}^z\nabla_H\cdot vd\xi\right)\right\|_\infty\|\partial_z\zeta\|_2+\|\nabla v\|_\infty\|\delta_l\nabla\zeta\|_2\right)\|\delta_l\Delta\zeta\|_2\\
\leq&C\left[\int_\Omega(|v|^2||\nabla^3u|^2+|\nabla v|^2||\nabla^2u|^2+|\nabla^2v|^2|\nabla u|^2+|u|^2|\nabla^3v|^2)d\textbf{x}\right]^{1/2}\|\delta_l\Delta\zeta\|_2\\
&+C\left(\|\partial_zT\|_{H^2}+\left\|\partial_z
\left(\int_{-h}^z\nabla_H\cdot vd\xi\right)\right\|_\infty\|u\|_{H^2}+\|v\|_{H^3}\|\delta_l\nabla\zeta\|_2\right)\|\delta_l\Delta\zeta\|_2\\
\leq&C(\|v\|_\infty\|u\|_{H^3}+\|\nabla v\|_\infty\|u\|_{H^2}+\|\nabla u\|_\infty\|v\|_{H^2}+\|u\|_\infty\|v\|_{H^3}\\
&+\|\partial_zT\|_{H^2}+\|\nabla v\|_\infty\|u\|_{H^2}+\|v\|_{H^3}\|\delta_l\nabla\zeta\|_2)\|\delta_l\Delta\zeta\|_2\\
\leq&C(\|v\|_{H^2}\|u\|_{H^3}+\|u\|_{H^2}\|v\|_{H^3}+\|\partial_zT\|_{H^2}+\|v\|_{H^3}\|
\delta_l\nabla\zeta\|_2)\|\delta_l\Delta\zeta\|_2\\
\leq&\sigma\|\delta_l\Delta\zeta\|_2^2+C_\sigma\|v\|_{H^3}^2\|\delta_l\nabla\zeta\|_2^2+C_\sigma(\|u\|_{H^3}^2
\|v\|_{H^2}^2+\|v\|_{H^3}^2\|u\|_{H^2}^2+\|\partial_zT\|_{H^2}^2),
\end{align*}
with a sufficiently small positive constant $\sigma$, and thus
\begin{align*}
&\frac{d}{dt}\|\delta_l\nabla\zeta\|_2^2+\frac{1}{R}\|\delta_l\nabla^2\zeta\|_2^2\\
\leq&C\|v\|_{H^3}^2\|\delta_l\nabla\zeta\|_2^2+C(\|u\|_{H^3}^2
\|v\|_{H^2}^2+\|v\|_{H^3}^2\|u\|_{H^2}^2+\|\partial_zT\|_{H^2}^2),
\end{align*}
from which it follows that
\begin{align*}
&\frac{d}{dt}(\|\delta_l\nabla\zeta\|_2^2t^2)+\frac{1}{R}\|\delta_l\nabla^2\zeta\|_2^2t^2\\
\leq&C\|v\|_{H^3}^2\|\delta_l\nabla\zeta\|_2^2t^2+C(\|u\|_{H^3}^2
\|v\|_{H^2}^2+\|v\|_{H^3}^2\|u\|_{H^2}^2\\
&+\|\partial_zT\|_{H^2}^2)t^2+2t\|\delta_l\nabla\zeta\|_2^2.
\end{align*}
Combining this inequality with (\ref{nnw4}), it follows
\begin{align*}
&\sup_{0\leq s\leq t}(\|\delta_l\nabla\zeta\|_2^2s^2)+\int_0^ts^2\|\delta_l\nabla^2\zeta\|_2^2ds\\
\leq&Ce^{C\int_0^t\|v\|_{H^3}^2ds}\int_0^t[(\|u\|_{H^3}^2
\|v\|_{H^2}^2+\|v\|_{H^3}^2\|u\|_{H^2}^2+\|\partial_zT\|_{H^2}^2)s^2+s\|\delta_l\nabla\zeta\|_2^2]ds\\
\leq&Ce^{C\int_0^t\|v\|_{H^3}^2ds}\int_0^t[(\|u\|_{H^3}^2
\|v\|_{H^2}^2+\|v\|_{H^3}^2\|u\|_{H^2}^2+\|\partial_zT\|_{H^2}^2)s^2+s\|\partial_z^2\nabla u\|_2^2]ds\\
\leq&Ce^{CK_0(t)}(tK_0(t)K_1(t)+K_0(t)t^2+K_1(t)),
\end{align*}
which, by Lemma \ref{lem0}, implies $t^2\partial_z\zeta\in L^\infty(0,t_0; H^1(\Omega))\cap L^2(0,t_0;H^2(\Omega))$ and
\begin{equation}\label{nnw6}
\sup_{0\leq s\leq t}(s^2\|\partial_z\zeta\|_{H^1}^2)+\int_0^ts^2\|\partial_z\zeta\|_{H^2}^2ds
\leq K_2'(t)
\end{equation}
for any $t\in(0,t_0)$.

Set
\begin{align*}
f_2(x,y,z,t)=&\left(\int_{-h}^z\nabla_H\cdot vd\xi\right)\partial_z\zeta+\nabla_H\cdot v\partial_zu
+\nabla_H\partial_zT\\
&-\partial_z[(v\cdot\nabla_H)u+(u\cdot\nabla_H)v-(\nabla_H\cdot v)u]-f_0k\times\zeta.
\end{align*}
Then it follows that
\begin{align*}
\|\nabla f_2\|_2^2\leq&C\int_\Omega\left[\left(\int_{-h}^h|\nabla^2v|d\xi\right)^2|\partial_z\zeta|^2+\left(\int
_{-h}^h|\nabla v|d\xi\right)^2|\nabla\partial_z\zeta|^2+|\partial_z\nabla^2T|^2\right.\\
&+|v|^2|\nabla^3u|^2+|\nabla v|^2|\nabla^2u|^2+|\nabla^2v|^2|\nabla u|^2+|u|^2|\nabla^3v|^2\bigg]dxdydz\\
\leq&C(\|\nabla^2v\|_2^2\|\partial_z\zeta\|_\infty^2+\|\nabla v\|_\infty^2\|\nabla\partial_z\zeta\|_2^2+\|\partial_zT\|_{H^2}^2+\|v\|_\infty^2\|\nabla^3u\|_2^2\\
&+\|\nabla v\|_\infty^2\|\nabla^2u\|_2^2+\|\nabla u\|_\infty^2\|\nabla^2v\|_2^2+\|u\|_\infty^2\|\nabla^3v\|_2^2)\\
\leq&C(\|\partial_z\zeta\|_{H^2}^2\|v\|_{H^2}^2+\|v\|_{H^3}^2\|\partial_z\zeta\|_{H^2}^2+\|\partial
_zT\|_{H^2}^2\\
&+\|v\|_{H^2}^2\|u\|_{H^3}^2+\|v\|_{H^3}^2\|u\|_{H^2}^2),
\end{align*}
and thus
\begin{align}
\int_0^t s^2\|\nabla f_2\|_2^2ds\leq&C\int_0^t[s^2(\|\partial_z\zeta\|_{H^2}^2\|v\|_{H^2}^2+\|v\|_{H^3}^2
\|\partial_z\zeta\|_{H^2}^2)\nonumber\\
&+s^2(\|v\|_{H^3}^2\|u\|_{H^2}^2+\|u\|_{H^3}^2\|v\|_{H^2}^2)+s^2\|
\partial_zT\|_{H^2}^2]ds\nonumber\\
\leq&C(K_0(t)K_2'(t)+tK_0(t)K_1(t)+t^2K_0(t))\nonumber\\
=&CK_0(t)(K_2'(t)+tK_1(t)+t^2):=K_2''(t). \label{nw7}
\end{align}
Recalling the definition of $f_2$, by (\ref{nw5}), one can easily see that $\zeta$ satisfies
\begin{equation}
\partial_t \zeta +L_1 \zeta =f_2. \label{nw8}
\end{equation}
Multiplying equation (\ref{nw8}) by $-(\delta_l^i)^*\delta_l^i\Delta\zeta, i=1,2$, thanks to Lemma \ref{lem2.0}, we have
\begin{align*}
&\frac{d}{dt}\int_\Omega|\delta_l^i\nabla\zeta|^2dxdydz+\int_\Omega\left(\frac{1}{R_1}|
\delta_l^i\nabla_H\nabla\zeta|^2+\frac{1}{R_2}|\delta_l^i\partial_z\nabla\zeta|^2\right)dxdydz\\
=&-\int_\Omega\delta_l^if_2\delta_l^i\Delta\zeta\leq\eta\int_\Omega|\delta_l^i\Delta\zeta|^2dxdydz
+C\int_\Omega|\delta_l^if_2|^2dxdydz,
\end{align*}
and thus
\begin{equation*}
\frac{d}{dt}\|\delta_l^i\nabla\zeta\|_2^2+\frac{1}{R}\|\delta_l^i\nabla^2\zeta\|_2^2\leq C\|\delta_l^if_2\|_2^2\leq C\|\partial_if_2\|_2^2,
\end{equation*}
with $R=R_1+R_2$, from which we obtain
\begin{align}
\frac{d}{dt}(t^2\|\delta_l^i\nabla\zeta\|_2^2)+\frac{1}{R}t^2\|\delta_l^i\nabla^2\zeta\|_2^2\leq
Ct^2\|\nabla f_2\|_2^2+2t\|\delta_l^i\nabla\zeta\|_2^2.\label{ex1}
\end{align}
Integrating (\ref{ex1}) with respect to $t$, by (\ref{nnw4}), (\ref{nw7}) and Lemma \ref{lem0}, we then obtain
\begin{align*}
&\sup_{0\leq s\leq t}(s^2\|\nabla_H\nabla\zeta\|_2^2)+\int_0^ts^2\|\nabla_H\nabla^2\zeta\|_2^2ds\\
\leq&C\int_0^ts^2\|\nabla f_2\|_2^2ds+C\int_0^ts\|\nabla^3u\|_2^2ds\leq C(K_1(t)+K_2''(t)).
\end{align*}
Combining this with (\ref{nnw6}) and using (\ref{nw7}), (\ref{nw8}), we obtain
\begin{equation}\label{kk1}
\sup_{0\leq s\leq t}(s^2\| \zeta\|_{H^2}^2)+\int_0^ts^2(\| \zeta\|_{H^3}^2+\|\partial_t\zeta\|_{H^1}^2)ds
\leq K_2(t)
\end{equation}
for all $t\in(0,t_0)$. This completes the proof.
\end{proof}

Next we prove the regularity properties of $\eta$ and $\theta$, that is the following:

\begin{proposition}\label{lem5.3}
Let $(v,T)$ be a strong solution to system (\ref{1.10})--(\ref{1.16}) in $\Omega\times(0,t_0)$. Let $\eta$ and $\theta$ be the functions given by (\ref{eta}) and (\ref{theta}), respectively. Then it holds that
\begin{equation*}
t\eta,t\theta\in L^\infty(0,t_0; H^2(\Omega))\cap L^2(0,t_0; H^3(\Omega)),
\quad t\partial_t\eta,t\partial_t\theta\in L^2(0,t_0; H^1(\Omega))
\end{equation*}
and
\begin{equation*}
\sup_{0\leq s\leq t}[s^2(\|\eta\|_{H^2}^2+\|\theta\|_{H^2}^2)]+\int_0^ts^2(\|\eta\|_{H^3}^2
+\|\theta\|_{H^3}^2+\|\partial_t\eta\|_{H^1}^2+\|\partial_t\theta\|_{H^1}^2)ds\leq K_3(t)
\end{equation*}
for any $t\in(0,t_0)$, where $K_3(t)$ is a bounded increasing function on $(0,t_0)$.
\end{proposition}

\begin{proof}
One can easily check that $\eta$ and $\theta$ satisfy
\begin{eqnarray}
&\partial_t\eta+L_1\eta=f_3,\label{k1}\\
&\partial_t\theta+L_1\theta=f_4,\label{k2}
\end{eqnarray}
where
\begin{eqnarray*}
f_3(x,y,z,t)=&-\nabla_H^\bot\cdot\Big[(v\cdot\nabla_H)u-\left(\int_{-h}^z\nabla_H\cdot vd\xi\right)\partial_zu+(u\cdot\nabla_H)v\\
&-(\nabla_H\cdot v)u\Big]+f_0(R_1T-\theta),\\
f_4(x,y,z,t)=&-\nabla_H\cdot\Big[(v\cdot\nabla_H)u-\left(\int_{-h}^z\nabla_H\cdot vd\xi\right)\partial_zu+(u\cdot\nabla_H)v\\
&-(\nabla_H\cdot v)u\Big]+f_0\eta+
R_1\left(\frac{1}{R_3}-\frac{1}{R_2}\right)\partial_z^2T\\
&- R_1\left[v\cdot\nabla_HT-\left(\int_{-h}^z\nabla_H\cdot vd\xi\right)\left(\partial_zT+\frac{1}{h}\right)\right].
\end{eqnarray*}
Direct calculations show
\begin{align*}
\|\nabla f_3\|_2^2\leq&C\int_\Omega\bigg[|v|^2|\nabla^3u|^2+|\nabla v|^2|\nabla^2u|^2+|\nabla^2v|^2|\nabla u|^2+|u|^2|\nabla^3v|^2\\
&+\left(\int_{-h}^h|\nabla v|d\xi\right)^2|\nabla^2\zeta|^2+\left(\int_{-h}^h|\nabla^2 v|d\xi\right)^2|\nabla\zeta|^2\\
&+\left(\int_{-h}^h|\nabla^3 v|d\xi\right)^2|\zeta|^2+|\nabla T|^2+|\nabla^3v|^2\bigg]dxdydz\\
\leq&C(\|v\|_\infty^2\|\nabla^3u\|_2^2+\|\nabla v\|_\infty^2\|\nabla^2u\|_2^2+\|u\|_\infty^2\|\nabla^3v\|_2^2\\
&+\|\nabla v\|_\infty\|\nabla^2\zeta\|_2^2+\|\nabla^2v\|_2^2\|\nabla\zeta\|_\infty^2+\|\nabla^3v\|_2^2
\|\zeta\|_\infty^2\\
&+\|\nabla^2v\|_2^2\|\nabla u\|_\infty^2+\|\nabla T\|_2^2+\|\nabla^3v\|_2^2)\\
\leq&C(\|v\|_{H^2}^2\|u\|_{H^3}^2+\|v\|_{H^3}^2\|u\|_{H^2}^2+\|v\|_{H^2}^2\|\zeta\|_{H^3}^2\\
&+\|v\|_{H^3}^2\|\zeta\|_{H^2}^2+\|T\|_{H^2}^2+\|v\|_{H^3}^2),
\end{align*}
and
\begin{align*}
\|\nabla f_4\|_2^2\leq&C\int_\Omega\bigg[|v|^2|\nabla^3u|^2+|\nabla v|^2|\nabla^2u|^2+|\nabla^2v|^2|\nabla u|^2+|u|^2|\nabla^3v|^2\\
&+\left(\int_{-h}^h|\nabla v|d\xi\right)^2|\nabla^2\zeta|^2+\left(\int_{-h}^h|\nabla^2 v|d\xi\right)^2|\nabla\zeta|^2\\
&+\left(\int_{-h}^h|\nabla^3 v|d\xi\right)^2|\zeta|^2+|\nabla^3v|^2+|\nabla \partial_z^2T|^2+|\nabla v|^2|\nabla T|^2+|v|^2|\nabla^2T|^2\\
&+\left(\int_{-h}^h|\nabla^2v|d\xi\right)^2(|\partial_zT|^2+1)
+\left(\int_{-h}^h|\nabla v|d\xi\right)^2|\nabla\partial_zT|^2\bigg]dxdydz\\
\leq&C(\|v\|_{H^2}^2\|u\|_{H^3}^2+\|v\|_{H^3}^2\|u\|_{H^2}^2+\|v\|_{H^2}^2\|\zeta\|_{H^3}^2
+\|v\|_{H^3}^2\|\zeta\|_{H^2}^2\\
&+\|v\|_{H^3}^2+\|\partial_zT\|_{H^2}^2+\|\nabla v\|_\infty^2\|\nabla T\|_2^2+\|v\|_\infty^2\|\nabla^2T\|_2^2\\
&+\|\partial_zT\|_\infty^2 \|\nabla^2v\|_2^2+\|\nabla v\|_\infty^2\|\nabla^2T\|_2^2)\\
\leq&C(\|v\|_{H^2}^2\|u\|_{H^3}^2+\|v\|_{H^3}^2\|u\|_{H^2}^2+\|v\|_{H^2}^2\|\zeta\|_{H^3}^2
+\|v\|_{H^3}^2\|\zeta\|_{H^2}^2\\
&+\|v\|_{H^3}^2+\|\partial_zT\|_{H^2}^2+\|v\|_{H^3}^2\|T\|_{H^2}^2
+\|v\|_{H^2}^2\|\partial_zT\|_{H^2}^2)\\
\leq&C(\|v\|_{H^2}^2\|u\|_{H^3}^2+\|v\|_{H^3}^2\|u\|_{H^2}^2+\|v\|_{H^2}^2\|\zeta\|_{H^3}^2
+\|v\|_{H^3}^2\|\zeta\|_{H^2}^2)\\
&+C(\|\partial_zT\|_{H^2}^2+\|v\|_{H^3}^2)(1+\|v\|_{H^2}^2+\|T\|_{H^2}^2).
\end{align*}
By the aid of (\ref{nnw4}) and (\ref{kk1}), it follows from the above two inequalities that
\begin{align*}
\int_0^ts^2\|\nabla f_3\|_2^2ds\leq&C\int_0^ts^2(\|v\|_{H^2}^2\|u\|_{H^3}^2+\|v\|_{H^3}^2\|u\|_{H^2}^2+\|v\|_{H^2}^2\|\zeta\|_{H^3}^2\\
&+\|v\|_{H^3}^2\|\zeta\|_{H^2}^2+\|T\|_{H^2}^2+\|v\|_{H^3}^2)ds\\
\leq&CK_0(t)(tK_1(t)+K_2(t)+t^3+t^2)=:K_3'(t)
\end{align*}
and
\begin{align*}
\int_0^ts^2\|\nabla f_4\|_2^2ds\leq&C\int_0^ts^2[(\|v\|_{H^2}^2\|u\|_{H^3}^2+\|v\|_{H^3}^2\|u\|_{H^2}^2
+\|v\|_{H^2}^2\|\zeta\|_{H^3}^2\\
&+\|v\|_{H^3}^2\|\zeta\|_{H^2}^2)+(\|\partial_zT\|_{H^2}^2+\|v\|_{H^3}^2)
(1+\|v\|_{H^2}^2+\|T\|_{H^2}^2)]ds\\
\leq&CK_0(t)(K_2(t)+tK_1(t)+K_0(t)t^2+t^2)=:K_3''(t).
\end{align*}
Using the above two inequalities, multiplying (\ref{k1}) and (\ref{k2}) by $-(\delta_l^i)^*\delta_l^i\Delta\eta$ and $-(\delta_l^i)^*\delta_l^i\Delta\theta$, $i=1,2,3$, respectively, and summing the resulting equations up, then using similar argument as for (\ref{ex1}) leads to
\begin{align*}
&\frac{d}{dt}[t^2(\|\delta_l^i\nabla\eta\|_2^2+\|\delta_l^i\nabla\theta\|_2^2)]+\frac{1}{R}t^2(
\|\delta_l^i\nabla^2\eta\|_2^2+\|\delta_l^i\nabla^2\theta\|_2^2)\\
\leq&Ct^2(\|\nabla f_3\|_2^2+\|\nabla f_4\|_2^2)+2t
(\|\delta_l^i\nabla\eta\|_2^2+\|\delta_l^i\nabla\theta\|_2^2)\\
\leq&Ct^2(\|\nabla f_3\|_2^2+\|\nabla f_4\|_2^2)+Ct
(\|\nabla^2\eta\|_2^2+\|\nabla^2\theta\|_2^2)\\
\leq&Ct^2(\|\nabla f_3\|_2^2+\|\nabla f_4\|_2^2)+Ct(\|u\|_{H^3}^2+\|T\|_{H^2}^2).
\end{align*}
Integrating this inequality with respect to $t$ and applying Lemma \ref{lem0}, then it follows from (\ref{nnw4}) and the estimates on $f_3,f_4$ that
\begin{align*}
&\sup_{0\leq s\leq t}[s^2(\|\nabla^2\eta\|_2^2+\|\nabla^2\theta\|_2^2)]+\int_0^ts^2(\|\nabla^3\eta\|_2^2
+\|\nabla^3\theta\|_2^2)ds\\
\leq&C(K_3'(t)+K_3''(t)+t^2K_0(t)+K_1(t)).
\end{align*}
By the aid of this inequality and using the estimates for $f_3,f_4$, equations (\ref{k1}) and (\ref{k2}), we obtain
\begin{equation*}
\sup_{0\leq s\leq t}[s^2(\|\eta\|_{H^2}^2+\|\theta\|_{H^2}^2)]+\int_0^ts^2(\|\eta\|_{H^3}^2
+\|\theta\|_{H^3}^2+\|\partial_t\eta\|_{H^1}^2+\|\partial_t\theta\|_{H^1}^2)ds\leq K_3(t)
\end{equation*}
for any $t\in(0,t_0)$. This completes the proof.
\end{proof}

Finally, we consider the regularity properties of $\bar v$.

\begin{proposition}\label{lem5.4}
Let $(v,T)$ be a strong solution to system (\ref{1.10})--(\ref{1.16}) in $\Omega\times(0,t_0)$. Let $\bar v$ be the vertical average of $v$ as it was defined in the Introduction.
Then
\begin{equation*}
\sqrt t\bar v\in L^\infty(0,t_0; H^3(\Omega))\cap L^2(0,t_0; H^4(\Omega)),
\quad \sqrt t\partial_t\bar v\in L^2(0,t_0; H^2(\Omega)),
\end{equation*}
and
\begin{equation*}
\sup_{0\leq s\leq t}(s\|\bar v\|_{H^3}^2)+\int_0^ts(\|\bar v\|_{H^4}^2+\|\partial_t\bar v\|_{H^2}^2)ds\leq K_4(t),
\end{equation*}
for any $t\in(0,t_0)$, where $K_4(t)$ is a bounded increasing function on $(0,t_0)$.
\end{proposition}

\begin{proof}
By equation (\ref{1.17}) and (\ref{1.18}), $\bar v$ satisfies
\begin{eqnarray}
&\partial_t\bar v-\frac{1}{R_1}\Delta_H\bar v+\nabla_H\bar p(x,y,t)=f_5,\label{bar1}\\
&\nabla_H\cdot\bar v=0, \label{bar2}
\end{eqnarray}
where
\begin{align*}
\bar p(x,y,t)=&p_s(x,y,t)-\frac{1}{2h}\int_{-h}^h
\int_{-h}^zT(x,y,\xi,t)d\xi dz,\\
f_5(x,y,t)=&-(\bar v\cdot\nabla_H)\bar v-\overline{(\tilde v\cdot\nabla_H)\tilde v+(\nabla_H\cdot\tilde v)\tilde v}-f_0k\times\bar v.
\end{align*}
It follows from the Sobolev embedding inequality that
\begin{align*}
\|\nabla^2_Hf_5\|_2^2\leq&\int_M|\nabla^2_H[(\bar v\cdot\nabla_H)\bar v+\overline{(\tilde v\cdot\nabla_H)\tilde v+(\nabla_H\cdot\tilde v)\tilde v}+f_0k\times\bar v]|^2dxdy\\
\leq&C\int_M\Big[\left|\nabla_H^2\left(\overline{(\tilde v\cdot\nabla_H)\tilde v+(\nabla_H\cdot\tilde v)\tilde v}\right)\right|^2+|\bar v|^2|\nabla_H^3\bar v|^2\\
&+|\nabla_H\bar v|^2|\nabla_H^2\bar v|^2+|\nabla_H^2\bar v|^2\Big]dxdy\\
\leq&\int_M\bigg[\left(\int_{-h}^h(|\tilde v||\nabla^3_H\tilde v|+|\nabla_H\tilde v||\nabla_H^2\tilde v|)d\xi\right)^2\\
&+|\bar v|^2|\nabla_H^3\bar v|^2+|\nabla_H\bar v|^2|\nabla_H^2\bar v|^2+|\nabla_H^2\bar v|^2\bigg]dxdy\\
\leq&C(\|\bar v\|_\infty^2\|\nabla_H^3\bar v\|_2^2+\|\nabla_H\bar v\|_\infty^2\|\nabla_H^2\bar v\|_2^2+\|\tilde v\|_\infty^2\|\nabla_H^3\tilde v\|_2^2\\
&+\|\nabla_H\tilde v\|_\infty^2\|\nabla_H^2\tilde v\|_2^2+\|\nabla_H^2\bar v\|_2^2)\\
\leq&C(\|v\|_{H^3}^2\|v\|_{H^2}^2+\|v\|_{H^2}^2).
\end{align*}
Note that $\bar v\in L^2(0,t_0; H^3(M))$ and $\partial_t\bar v\in L^2(0,t_0; H^1(M))$. Applying the operator $\nabla_H$ to equation (\ref{bar1}), multiplying the resulting equation by $-(\delta_l^i)^*\delta_l^i\nabla_H\Delta_H\bar v, i=1,2$ and using (\ref{bar2}), then similar argument to that for (\ref{ex1}) leads to
\begin{align*}
&\frac{d}{dt}(t\|\delta_l^i\Delta_H\bar v\|_2^2)+\frac{1}{R_1}t\|\delta_l^i\nabla_H\Delta_H\bar v\|_2^2\\
\leq&Ct\|\nabla_H^2 f_5\|_2^2+\|\delta_l^i\Delta_H\bar v\|_2^2
\leq C(t\|\nabla_H^2 f_5\|_2^2+\|\nabla_H\Delta_H\bar v\|_2^2).
\end{align*}
Integrating this inequality with respect to $t$ and applying Lemma \ref{lem0}, then one has
\begin{align*}
&\sup_{0\leq s\leq t}(s\|\nabla_H\Delta_H\bar v\|_2^2)+\int_0^ts\|\nabla_H^2\Delta_H\bar v\|_2^2ds\\
\leq&C\int_0^t(s\|\nabla_H^2f_5\|_2^2+\|\bar v\|_{H^3}^2)ds \leq CK_0(t)(K_0(t)t+t^2+1).
\end{align*}
By the aid of this inequality, applying the $L^2$ theory of Stokes equations and using the estimate on $f_5$, we obtain
\begin{equation*}
\sup_{0\leq s\leq t}(s\|\bar v\|_{H^3}^2)+\int_0^ts(\|\bar v\|_{H^4}^2+\|\partial_t\bar v\|_{H^2}^2)ds\leq K_4(t)
\end{equation*}
for any $t\in(0,t_0)$. This completes the proof.
\end{proof}

\section*{Acknowledgments}
{The work of C.C.  work is  supported in part by NSF grant DMS-1109022. The work of E.S.T. is supported in part by the Minerva Stiftung/Foundation, and  by the NSF
grants DMS-1009950, DMS-1109640 and DMS-1109645.}
\par

\end{document}